\documentclass[11pt]{article}
\usepackage[T1]{fontenc}
\usepackage{amsmath,amsfonts,amsthm,mathrsfs,amssymb,cite}
\usepackage{graphicx}
\usepackage{subfigure}
\usepackage{placeins}
\usepackage{color}
\usepackage{indentfirst}
\usepackage[colorlinks=true]{hyperref}

\usepackage[dvipsnames, svgnames, table]{xcolor}
\usepackage{tikz} 
\usetikzlibrary{patterns}
\usepackage{tikz-3dplot}

\numberwithin{equation}{section}
\newcommand{\R}{{\mathbb R}}

\newcommand{\C}{{\mathbb C}}
\newcommand{\s}{{\mathbb S}}

\newcommand{\no}{\nonumber}
\newcommand{\be}{\begin{eqnarray}}
\newcommand{\ben}{\begin{eqnarray*}}
\newcommand{\en}{\end{eqnarray}}
\newcommand{\enn}{\end{eqnarray*}}

\newtheorem{theorem}{Theorem}[section]
\newtheorem{lemma}[theorem]{Lemma}
\newtheorem{corollary}[theorem]{Corollary}

\newtheorem{remark}[theorem]{Remark}
\newtheorem{proposition}[theorem]{Proposition}

\definecolor{rot}{rgb}{1.000,0.000,0.000}
\definecolor{rot1}{rgb}{0.000,0.000,0.000}

\begin{document}
\renewcommand{\theequation}{\arabic{section}.\arabic{equation}}

\begin{titlepage}
\title{\bf A Paley-Wiener-Type Theorem with Applications to Inverse Medium Scattering for the Schr\"odinger Equation}
\author{Guanghui Hu\thanks{School of Mathematical Sciences and LPMC, Nankai University, Tianjin 300071, China. Email: \texttt{ghhu@nankai.edu.cn}}, 
Yubiao Zhang\thanks{School of Mathematics, Jilin University, Changchun 130012, China. Email: \texttt{yubiao.zhang.math@outlook.com}}, 
Tianyi Zheng\thanks{School of Mathematical Sciences and LPMC, Nankai University, Tianjin 300071, China. Email: \texttt{9820250139@nankai.edu.cn}} 
and Jun Zou\thanks{(Corresponding author) Department of Mathematics, The Chinese University of Hong Kong, Shatin, N.T., Hong Kong
	SAR. Email: \texttt{zou@math.cuhk.edu.hk}}}
\date{}
\end{titlepage}
\maketitle
\vspace{.2in}

\begin{abstract}
This paper develops a new uniqueness framework for inverse scattering problems for the time-harmonic Schr\"odinger equation using multi-frequency data. Our approach is based on a Paley-Wiener-type theorem that characterizes the support of the potential through the asymptotic growth of a generalized Fourier integral along imaginary frequencies, allowing for nontrivial amplitude functions.

Within this framework, we obtain two major results on identifying the shape of the support of the potential in three dimensions, for both the plane-wave incidence and point-source incidence, respectively. For plane-wave incidence, we prove that the location and shape of the convex compact support of an essentially bounded potential are uniquely determined by fixed-angle far-field data from two opposite incident directions or by backscattering far-field data, provided that the potential is real-valued and sign-definite near the boundary of its support. For point-source incidence, we show that near-field backscattering data, namely the scattered field evaluated at the source point as the source ranges over a sufficiently large sphere, uniquely determine the location and shape of the convex support under the same assumptions on the potential. The proofs rely on meromorphic extension and the Paley-Wiener-type theorem.

\vspace{0.5cm}
\noindent

\vspace{.2in} {\bf Keywords}: Uniqueness, inverse medium scattering, multi-frequency data, Schr\"odinger equation, fixed-angle scattering, backscattering, Paley-Wiener-type theorem, shape identification. 

\vspace{.2in} {\bf 2020 Mathematics Subject Classification}: 35R30, 74B05, 78A46.
\end{abstract}

\section{Introduction}
\label{sec:introduction}

Let \(q\in L^\infty(\mathbb R^3)\) satisfy \(\operatorname{Im} q \geq 0\) and \(\operatorname{supp}(q)\subset B_R\) for some $R>0$. 
Consider the stationary Schr\"odinger equation in three dimensions:
\begin{equation}\label{eqs}
\Delta u + \kappa^2 u + q u = 0 ~\text{in} ~ \mathbb R^3,
\end{equation}
where \(\kappa>0\) is the wavenumber. The total field
\(u = u^{\mathrm{in}} + u^{\mathrm{sc}}\) consists of an incident plane wave of the form
\(u^{\mathrm{in}}(x)=e^{i\kappa x\cdot d}\) with propagation direction \(d\in\mathbb S^2\) and a scattered wave \(u^{\mathrm{sc}}\). The scattered field is required to satisfy the Sommerfeld
radiation condition
\begin{equation}\label{eq:radiation}
\lim_{r\to+\infty} r\Bigl(\frac{\partial u^{\mathrm{sc}}}{\partial r}-i\kappa u^{\mathrm{sc}}\Bigr)=0, ~ r=|x|,
\end{equation}
uniformly in all directions \(\hat x = x/|x|\).  

Any function $v$ in the space $H^2_{\mathrm{loc}}(\mathbb R^3) \cap C^1\big( \mathbb R^3 \backslash B_R \big)$ is said to be a solution to the scattering problem (1.1)--(1.2) if $v$ and $v^{\mathrm{sc}}:=v-e^{i\kappa x\cdot d}$ satisfy \eqref{eqs} and \eqref{eq:radiation}, respectively. By Theorem \ref{th:sw}, for each $d \in \mathbb S^2$ and $\kappa >0$, the scattering problem \eqref{eqs} and \eqref{eq:radiation} has a unique solution, denoted by $u(\cdot; d, \kappa)$. 

By Corollary \ref{Corollary:FarField-20261005}, the following asymptotic behavior holds for some \(u^{\infty}(\cdot; d, \kappa) \) in $ C( \mathbb S^2 )$: 
\begin{equation}\label{eq:farfield}
u^{\mathrm{sc}}(x;d,\kappa) = \frac{e^{i\kappa |x|}}{|x|}u^\infty(\hat x;d,\kappa)+ \mathcal{O}\Bigl(\frac1{|x|^2}\Bigr)
~\text{ as } |x|\to+\infty,
\end{equation}
uniformly in all directions \(\hat x = x/|x|\). The function \(u^\infty(\hat x;d,\kappa)\) is called the \emph{far-field pattern} or \emph{scattering amplitude}.

The inverse scattering problem considered in this paper is to recover information about the support of \(q\) from measurements of \(u^\infty(\hat x;d,\kappa)\). More precisely, we consider the following problem:
\begin{itemize}
\item[\textbf{(IP)}] \textbf{Shape identification}: determine the location and shape of \(\operatorname{supp}(q)\).
\end{itemize}

In inverse scattering theory, several types of measurement data are commonly employed for plane wave incidence:

\begin{enumerate}
\item[\textbf{(a)}] \textit{Fixed-energy scattering data}: $u^{\infty}(\theta;d,\kappa_0)$ known for all $d,\theta\in\mathbb{S}^2$ at a fixed wavenumber $\kappa_0>0$.
\item[\textbf{(b)}] \textit{Fixed-angle scattering data}: $u^{\infty}(\theta;d_0,\kappa)$ known for a fixed incident direction $d_0 \in \mathbb{S}^{2}$, all observation directions $\theta\in\mathbb{S}^2$, and all wavenumbers $\kappa>0$.
\item[\textbf{(c)}] \textit{Backscattering data}: $u^{\infty}(-\theta;\theta,\kappa)$ known for all incident directions $\theta \in \mathbb{S}^{2}$ and all wavenumbers $\kappa>0$.
\end{enumerate}

Analogous inverse problems can be formulated using measurement data generated by point-source incidence. 
Previous research has established several important results for these measurement types. For fixed-energy measurements, the data are equivalent to Dirichlet-to-Neumann (DtN) measurements, and the problem was first solved in \cite{SU87} for smooth, compactly supported, real-valued potentials. For fixed-angle measurements, the results in \cite{M18, R01, S08} show that singularities of the potential $q$ can be recovered. Subsequently, the results in \cite{RS20-1,RS20-2} established that smooth, compactly supported, real-valued potentials can be recovered from fixed two-angle scattering data, while \cite{MMS21} extended these results to the case of a first-order perturbation.
For backscattering measurements, the first generic uniqueness results in the frequency-domain setting were obtained in \cite{ER89, ER91}; analogous time-domain results were later proved in \cite{S90, S92}. The work in \cite{GU93} established that the principal symbol of $q$ (which characterizes its major singularities) can be recovered from backscattering data using time-domain methods. Similar results in the frequency domain were subsequently obtained in \cite{BFRV13, M21, OPS01, R07, RV05, S08}.

To the best of our knowledge, problem (IP) has not previously been considered in the literature. In this paper, we focus on the three-dimensional inverse problem \textbf{(IP)} for complex-valued potentials using measurement data of type \textbf{(b)} or \textbf{(c)} with plane wave incidence or point source incidence. By meromorphic extension (see Theorems \ref{th:sw} and \ref{th:ps}), knowledge of the measurement data for all $\kappa > 0$ is equivalent to knowledge of the data for $\kappa$ in a finite open interval. Hence, we use multi-frequency data over a finite frequency interval $\kappa \in (\kappa_{\min}, \kappa_{\max})$.

To state our main results, we make the following assumptions on the potential: 
\begin{itemize}
    \item[ ] \textbf{Assumption (H):} The potential $q \in L^\infty(\mathbb{R}^3)$ satisfies $\operatorname{Im} q \geq 0$ and has convex compact support $\overline{D} := \operatorname{supp}(q) \subset B_R$ (for some domain $D$). Moreover, $q$ keeps a uniform sign around the boundary of its support (see Figure \ref{fig:convex-domain}), i.e., there is an open set $U\supset \partial D$ such that
    \begin{align}\label{UniformSignConditionOnPotential-20261003}
       \text{ either }~ q \geq 0  \text{ in } U
       ~\text{ or }~
       q \leq 0  \text{ in } U. 
    \end{align}
\end{itemize}

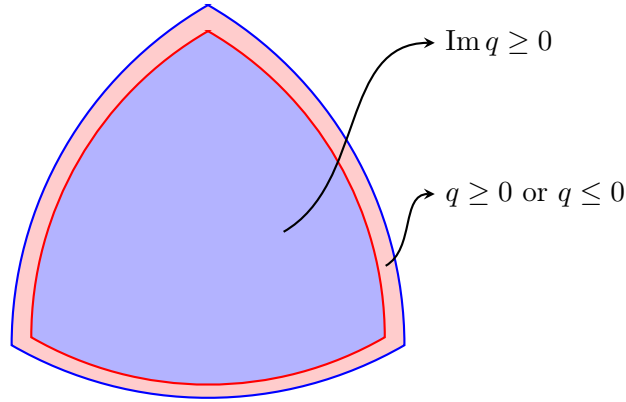
\begin{figure}[htbp]
    \centering
\begin{tikzpicture}

\def\R{3}
\pgfmathsetmacro{\a}{\R*sqrt(3)}

\coordinate (A) at (0,\R);
\coordinate (B) at ({-\R*cos(30)}, {-\R*sin(30)});
\coordinate (C) at ({\R*cos(30)}, {-\R*sin(30)});

\def\mypath{
    (A) arc[start angle=120, delta angle=60, radius=\a]
         arc[start angle=240, delta angle=60, radius=\a]
         arc[start angle=0,   delta angle=60, radius=\a]
         -- cycle
}

\filldraw[fill=red!20, draw=blue, thick] \mypath;

\begin{scope}[scale=0.9, shift={(0,-0.05)}]

\coordinate (A) at (0,\R);
\coordinate (B) at ({-\R*cos(30)}, {-\R*sin(30)});
\coordinate (C) at ({\R*cos(30)}, {-\R*sin(30)});

\def\mypath{
    (A) arc[start angle=120, delta angle=60, radius=\a]
         arc[start angle=240, delta angle=60, radius=\a]
         arc[start angle=0,   delta angle=60, radius=\a]
         -- cycle
}

    \filldraw[fill=blue!30, draw=red, thick] \mypath;
\end{scope}

\draw[->, thick, >=stealth] (1,0) to[out=30, in=180] (3,2.5) node[right] {$\operatorname{Im} q \geq 0$};
\draw[->, thick, >=stealth] (2.35,-0.45) to[out=30, in=180] (3,0.5) node[right] {$q\geq0$ or $q\leq0$};

\end{tikzpicture}

    \caption{Illustration of the convex support of $q$, with the sign-definiteness condition imposed near its boundary.}
    \label{fig:convex-domain}
\end{figure}

\begin{remark}
Unlike the existing results discussed above, our results allow \(q\) to be complex-valued, as long as \(\operatorname{Im} q \ge 0\). We use this condition to establish well-posedness of the scattering problem \eqref{eqs}-\eqref{eq:radiation} (see Step 3 in the proof of Theorem \ref{th:sw}). Without this condition, we still do not know how to show this result. Furthermore, near the boundary $\partial D$, we additionally assume $q$ to be real-valued and sign-definite. This requirement stems from the use of our Paley-Wiener-type theorem (see Theorem~\ref{thm-growth-support}).
\end{remark}

We can now state our main theorems. For the plane wave incidence:
\begin{equation}\label{ps}
u^{\mathrm{in}}(x; d, \kappa) := e^{i \kappa x\cdot d},~ x\in \mathbb R^3
\quad\text{where }d\in\mathbb{S}^2 \text{ and } \kappa >0,
\end{equation}
we denote the far field pattern by $u^\infty(\theta; d, \kappa)$.

\begin{theorem}[Plane wave incidence]\label{TH:SE}
    Suppose that $\mathrm{Assumption}$ $\mathrm{\mathbf{(H)}}$ holds.
Let $0 < \kappa_{\min} < \kappa_{\max} < \infty$. Then, the shape and location of $\overline{D}$ are uniquely determined by either of the following sets of measurement data:
\begin{itemize}
\item[$\mathrm{(i)}$] \textit{Fixed-angle scattering}: $u^\infty(\theta; d, \kappa)$ for all $\theta\in \mathbb{S}^2$, $\kappa\in (\kappa_{\min}, \kappa_{\max})$, and $d = \pm d_0$ for some fixed $d_0 \in \mathbb{S}^2$;
\item[$\mathrm{(ii)}$] \textit{Backscattering}: $u^\infty(\theta; -\theta, \kappa)$ for all $\theta\in \mathbb{S}^2$ and $\kappa\in (\kappa_{\min}, \kappa_{\max})$.
\end{itemize}
\end{theorem}

Next, we consider point-source incidence. Let $y \in \mathbb R^3$ and $\kappa > 0$. Consider the Schr\"odinger equation with an incident field generated by a point source at $y$:
\begin{eqnarray}\label{psiw}
    \left\{
    \begin{array}{r@{~}l}
        \big( \Delta+\kappa^2+q(x) \big)u(x)&=-\delta(x-y),  ~x \in \mathbb R^3,
        \vspace{0.5em}\\
        u&=u^{\mathrm{in}}+u^{\mathrm{sc}},
        \vspace{0.5em}\\
        u^{\mathrm{in}}(x)&=\dfrac{e^{i\kappa|x-y|}}{4\pi|x-y|}, ~ x \in \mathbb R^3\setminus \{y\},
        \vspace{0.5em}\\
        \lim\limits_{r\to+\infty}r(\partial_r-i\kappa)u^{\mathrm{sc}}&=0, ~ r=|x|.
    \end{array}
    \right.
\end{eqnarray}
Any function $u$ in the space $H^2_{\mathrm{loc}}(\mathbb R^3 \setminus \{y\}) \cap C^1\big( \mathbb R^3 \setminus (B_R \cup \{y\}) \big)$ is said to be a solution to the scattering problem (1.6) if $u$ satisfies \eqref{psiw}. By Theorem \ref{th:ps}, for each $y \in \mathbb R^3$ and $\kappa >0$, the scattering problem \eqref{psiw} has a unique solution. We denote the scattered and total fields by \(u^{\mathrm{sc}}(x; y, \kappa)\) and \(u(x;y,\kappa)\), respectively, to make their dependence on the source position \(y\) explicit. 

Since the total field $u(x;y,\kappa)$ is singular at $x = y$, we measure the scattered field $u^{\mathrm{sc}}(y;y,\kappa)$ at the source position for all $|y| = R$, and these measurements can be regarded as near-field backscattering data.

\begin{theorem}[Point source incidence]\label{TH:near} 
    Suppose that $\mathrm{Assumption}$ $\mathrm{\mathbf{(H)}}$ holds. 
Let $0 < \kappa_{\min} < \kappa_{\max} < \infty$. Then, the shape and location of $\overline{D}$ are uniquely determined by the near-field backscattering data $\big\{ u^{\mathrm{sc}}(y; y, \kappa) ~:~ |y| = R, ~\kappa \in (\kappa_{\min}, \kappa_{\max}) \big\}$. 
\end{theorem}

\bigskip
\noindent\textbf{Novelties and methods.} The main novelties in this paper are summarized as follows: 
\begin{itemize}
    \item[(a1)] We develop a new mathematical method for reconstructing the convex support of a nonsmooth, complex-valued potential under both plane-wave and point-source incidence (see Theorems \ref{TH:SE} and \ref{TH:near} as well as their proofs). 
    \item[(a2)] We establish a Paley-Wiener-type theorem (Theorem \ref{thm-growth-support}) that may have further applications to other inverse scattering problems. 
\end{itemize}
The proofs of Theorems \ref{TH:SE} and \ref{TH:near} follow the same approach, which we summarize below. 
First, the total fields and the measurements are extended meromorphically to complex wavenumbers (see Theorems \ref{th:sw} and \ref{th:ps}). Then, we establish the leading asymptotic expansions of the total fields and the measurements along the imaginary axis of complex wavenumbers (see Theorems \ref{lem3.3} and \ref{th:ps}). Finally, applying the Paley-Wiener-type theorem (Theorem \ref{thm-growth-support}) to the corresponding integral representations yields information about the support of the potential to be recovered. 

\bigskip
\noindent\textbf{Organization.}
The paper is structured as follows: Section \ref{sec:Prelimilary lemmas} presents some key lemmas, including the Paley-Wiener-type theorem (Theorem \ref{thm-growth-support}). Sections \ref{sec:proof-of-theorem-se} and \ref{sec:proof of theorem near field} contain the uniqueness proofs for Theorems \ref{TH:SE} and \ref{TH:near}, respectively. Finally, Section \ref{sec:Concluding remarks} presents concluding remarks and directions for future work.

\section{A Paley-Wiener-Type Theorem}
\label{sec:Prelimilary lemmas}

The celebrated Paley-Wiener theorem characterizes the support of a compactly supported distribution in terms of the growth of its Fourier transform. There are various versions of the Paley-Wiener theorem. The following version is taken from \cite[Theorem 7.3.1, p. 181]{Hormander}; see also \cite[Chapter III D, pp.~30--37]{PK} for the one-dimensional case.

\begin{proposition}[Paley-Wiener theorem]\label{PW}
Let $K$ be a convex compact subset of $\mathbb{R}^n$ with supporting function $H$ (i.e., $H(\xi) := \displaystyle\sup_{x\in K} x \cdot \xi$, $\xi \in \mathbb R^n$). Let  $u$ be a distribution compactly supported in $\mathbb R^n$. 
Then, u is supported in $K$ if and only if for some $C>0$ and $N \in \mathbb N$,
$$
|\hat{u}(\zeta)|\leq C(1+|\zeta|)^Ne^{H(\operatorname{Im}\zeta)},
\quad \zeta\in\mathbb{C}^n,
$$
where $\hat u$ denotes the Fourier transform of $u$ (with the extension to be a function over \(\mathbb C^n\)). 
\end{proposition}

For a compactly supported function \(u\in L^2(\mathbb R^n)\), Proposition \ref{PW} yields the following characterization of the supporting hyperplane in a direction \(e\in\mathbb{S}^{n-1}\) in terms of the growth of $\hat u$ along the lines \(\xi+i\eta e\), with \(\xi\in\mathbb R^n\):   
\begin{equation}\label{growth-pw}
\sup_{x\in\operatorname{supp}u}x\cdot e
=
\sup_{\xi\in\mathbb R^n}
\limsup_{\eta\to+\infty}
\frac{\ln|\hat u(\xi+i\eta e)|}{\eta}.
\end{equation}
For a proof of \eqref{growth-pw}, see Proposition \ref{ndpw} in Appendix \ref{sec:proof-of-growth-pw}.

In this paper, we need a Paley-Wiener-type theorem for integral operators of the following form: 
\begin{eqnarray}\label{FourierIntegralOperator-202601002}
    F(k) := \int_{\mathbb R^n} e^{-i k \varphi(x) } \mathcal A(k,x) f(x) dx, ~ k \in \mathbb C \text{ with }  \operatorname{Im} k \gg 1,
\end{eqnarray}
where $\varphi$ and $\mathcal A$ are suitable phase and amplitude functions, respectively. This includes the Fourier transform as a special case (where  $\mathcal A \equiv 1$ and $\varphi(x)=x \cdot \xi$ with $\xi \in \mathbb R^n$). 
The aforementioned generalized result is stated below. 

\begin{theorem}
    \label{thm-growth-support}
    Let  $M,r>0$. Let $\varphi \in L^1_{\mathrm{loc}}(\mathbb R^n) \cap C^1(B_r)$ be a real-valued function satisfying
    \begin{align}\label{PhaseFunction-20261003}
        \inf_{x \in B_r}  | \nabla \varphi(x) | > 0.
    \end{align}
Take a bounded function $\mathcal{A}$ from $\mathbb{C}_{M}^{+}$ to $L^{\infty}(\mathbb{R}^{n})$ (with $\C_{M}^+:= \big\{ k\in \C ~:~ \operatorname{Im} k \geq M \big\}$) such that
    \begin{eqnarray}
        \label{1020-1}
        \lim_{\eta \rightarrow+\infty} \mathcal A(i\eta, \cdot)|_{ B_r }=1
        \;\;\mbox{strongly in}\;\; L^\infty( B_r ).
    \end{eqnarray}
    Let $f\in L^1(\mathbb R^n) \setminus\{0\}$  with $K := \mathrm{supp}(f)\subset B_r$. 
    Assume that $f$ has a fixed sign in a neighborhood of $\partial K$ (i.e., either $f\geq 0$ around $\partial K$ or $f\leq 0$ around $\partial K$). Then, it holds that 
    \begin{eqnarray*}
        \sup_{ x \in \mathrm{supp}(f) }  \varphi(x)
        = \lim_{\eta\rightarrow+\infty} \frac{ \ln |F(i\eta)|}{ \eta },
    \end{eqnarray*}
    where $F$ is given by \eqref{FourierIntegralOperator-202601002}.
\end{theorem}

\begin{remark}\label{rmk:comparison}
The above Paley-Wiener-type theorem (Theorem \ref{thm-growth-support}) extends the classical version in the following two respects:
\begin{enumerate}    
    \item \textbf{Phase and amplitude functions:} 
    While the classical theorem considers the standard Fourier transform with the kernel $e^{-i x \cdot \xi}$, our result incorporates a phase function $\varphi$ and an amplitude function $\mathcal A$ satisfying \eqref{PhaseFunction-20261003} and \eqref{1020-1}, respectively. This allows the analysis of more general integral operators beyond the classical Fourier transform.
    
    \item \textbf{Asymptotic behavior:}
    In the classical identity \eqref{growth-pw}, one needs to take the supremum over all $\xi\in\mathbb R^n$, i.e. the growth of $\hat u$ is considered along all lines $\xi+i\eta e$ with arbitrary real part $\xi$. In contrast, Theorem \ref{thm-growth-support} uses only the purely imaginary axis $k=i\eta$, and no additional supremum over $\xi\in\mathbb R^n$ is required. In addition, our result establishes the \textit{existence of the limit} (not just limsup) under the additional fixed-sign condition near the support boundary, which compensates for the role of the generalized amplitude.
\end{enumerate}
\end{remark}

To prove Theorem \ref{thm-growth-support}, we need the following auxiliary result. 

\begin{lemma}\label{an-unique-continu-L1-weight}
    Suppose that $\varphi$, $\mathcal A$, and $f$ satisfy the same hypotheses of Theorem \ref{thm-growth-support}. Write $\tau := \displaystyle\sup_{ x\in \mathrm{supp}(f)}  \varphi(x)$.  
Then, it holds that 
\begin{eqnarray}\label{bdd-0830-2}
  \liminf_{\eta\rightarrow+\infty} \Big| 
      e^{-\eta(\tau-\varepsilon)} 
      \int_{ \{ \tau -\varepsilon < \varphi(x) \leq \tau \} } 
      e^{\eta \varphi(x) } \mathcal A(i\eta, x) f(x) dx 
  \Big| 
  =+\infty
  ~\text{ for each }~
  \varepsilon > 0.
\end{eqnarray}
\end{lemma}

\begin{proof}
    First of all, it follows from \eqref{PhaseFunction-20261003} that when $|c_1- c_2| \ll 1$, the two level sets $\{ \varphi = c_1 \}$ and $\{ \varphi = c_2 \}$ restricted in $B_r$ are very close to each other.  Since $f$ keeps a uniform sign  in a  neighborhood of  the boundary of supp$(f)$ with $\tau= \displaystyle\sup_{ x\in \mathrm{supp}(f)}  \varphi(x)$, we determine that for some sufficiently small $\varepsilon_0  > 0$, 
    \begin{align*}
    \text{either } f \geq 0 \text{ on } \{ \tau -\varepsilon_0 < \varphi(x) \leq \tau \}
    ~\text{ or }~
    f \leq 0 \text{ on } \{ \tau -\varepsilon_0 < \varphi(x) \leq \tau \}.
\end{align*}
Furthermore, \(f\) is not identically zero on \(\{ \tau -\delta < \varphi(x) \leq \tau \}\) for any \(\delta \in (0, \varepsilon_0]\); otherwise, this would contradict the definition of \(\tau\) as the supremum of \(\varphi\) on \(\operatorname{supp}(f)\). 
Without loss of generality, suppose that 
\begin{eqnarray}\label{1045-1}
f \geq 0 \text{ on } \{ \tau -\varepsilon_0 < \varphi(x) \leq \tau \}
~\text{ or }~
\int_{ \{ \tau -\delta < \varphi(x) \leq \tau \} }    |f(x)| dx > 0
~\text{ for each }
\delta \in (0,\varepsilon_0].
\end{eqnarray}
By \eqref{1020-1}, we can choose $\eta>0$ sufficiently large such that
\ben
\| \mathcal A(i\eta, \cdot) - 1 \|_{L^\infty( B_r )} < 1/2. 
\enn
Fix an arbitrary \(\varepsilon\in(0,\varepsilon_{0}]\). Using the first inequality in \eqref{1045-1}, and the preceding bound on \(\mathcal{A}(i\eta,\cdot)-1\), we obtain, for sufficiently large \(\eta>0\),
\begin{eqnarray*}
&&\Big| \int_{ \{ \tau -\varepsilon < \varphi(x) \leq \tau \} } 
e^{\eta \varphi(x) } \mathcal A(i\eta, x) f(x) dx \Big|
= \Big| 
\int_{ \{ \tau -\varepsilon < \varphi(x) \leq \tau \} } 
e^{\eta \varphi(x) } \big( 1 + ( \mathcal A(i\eta,x) - 1 ) \big) f(x) dx
 \Big|
 \\
&\geq&  \Big| 
\int_{ \{ \tau -\varepsilon < \varphi(x) \leq \tau \} }     e^{\eta \varphi(x) }   f(x) dx
\Big|
-  \| \mathcal A(i\eta,\cdot) - 1 \|_{L^\infty( B_r )}
\int_{ \{ \tau -\varepsilon < \varphi(x) \leq \tau \} }    e^{\eta \varphi(x)}  |f(x)| dx
\\
&\geq& 
\frac{1}{2}  
\int_{ \{ \tau -\varepsilon < \varphi(x) \leq \tau \} }     e^{\eta \varphi(x) }   |f(x)| dx
= \frac{1}{2} e^{\eta(\tau-\varepsilon)}
\int_{ \{ \tau -\varepsilon < \varphi(x) \leq \tau \} }     e^{\eta ( \varphi(x) - (\tau-\varepsilon) ) }   |f(x)| dx.
\end{eqnarray*}
This, together with the second inequality in \eqref{1045-1}, yields
\begin{align}\label{LowerBoundOfIntegral-SmallEpsilon}
    &\liminf_{\eta\rightarrow+\infty} \Big| e^{-\eta(\tau-\varepsilon)} 
    \int_{ \{ \tau -\varepsilon < \varphi(x) \leq \tau \} } 
    e^{\eta \varphi(x) } \mathcal A(i\eta, x) f(x) dx
    \Big|
    \nonumber\\
    \geq& \frac{1}{2} \liminf_{\eta\rightarrow+\infty}  
    \int_{ \{ \tau -\varepsilon < \varphi(x) \leq \tau \} }     e^{\eta ( \varphi(x) - (\tau-\varepsilon) ) }   |f(x)| dx
    = + \infty,
\end{align}
which proves \eqref{bdd-0830-2} when $\varepsilon \in (0, \varepsilon_0]$. 

When $\varepsilon  > \varepsilon_0$, observe that
\begin{align*}
    &\Big| \int_{ \{ \tau -\varepsilon < \varphi(x) \leq \tau \} } 
    e^{\eta \varphi(x) } \mathcal A(i\eta, x) f(x) dx \Big|
    = \Big| \int_{ \{ \tau -\varepsilon_0 < \varphi(x) \leq \tau \} }  \cdots dx + \int_{ \{ \tau -\varepsilon < \varphi(x) \leq \tau -\varepsilon_0 \} }  \cdots dx \Big|
    \nonumber\\
    \geq&  \Big| \int_{ \{ \tau -\varepsilon_0 < \varphi(x) \leq \tau \} } 
    e^{\eta \varphi(x) } \mathcal A(i\eta, x) f(x) dx \Big|
    - e^{\eta (\tau-\varepsilon_0)}  \| \mathcal A(i\eta, \cdot)  \|_{L^\infty( B_r )}  
    \| f\|_{L^1( B_r )}.
\end{align*}
Then, we  obtain
\begin{align*}
     \Big| \int_{ \{ \tau -\varepsilon < \varphi(x) \leq \tau \} }  e^{\eta \varphi(x)} \mathcal A(i\eta,x) f(x) dx \Big|
     \geq&  e^{\eta (\tau-\varepsilon_0)}  \bigg(
     e^{-\eta(\tau-\varepsilon_0)} 
     \int_{ \{ \tau -\varepsilon_0 < \varphi(x) \leq \tau \} } 
     e^{\eta \varphi(x) } \mathcal A(i\eta, x) f(x) dx
     \nonumber\\
     &      - \| \mathcal A( i\eta,\cdot)  \|_{L^\infty(B_r)}  
     \| f\|_{L^1(B_r)}
     \bigg).
\end{align*}
This, along with \eqref{LowerBoundOfIntegral-SmallEpsilon} (with $\varepsilon$ replaced by $\varepsilon_0$),  yields \eqref{bdd-0830-2} when $\varepsilon > \varepsilon_0$. 
This completes the proof. 
\end{proof}

We are now in a position to prove Theorem \ref{thm-growth-support}. 

\begin{proof}[Proof of Theorem \ref{thm-growth-support}]
By the definition of $F$ in \eqref{FourierIntegralOperator-202601002}, for sufficiently large $\eta$,
\begin{eqnarray*}
 |F(i\eta)| &\leq&  \int_{ \mathrm{supp}(f) } e^{\eta \varphi(x) } |\mathcal{A}(i\eta, x)| \, |f(x)| dx
   \leq  \exp\Big( \eta \sup_{ x\in \mathrm{supp}(f)} \varphi(x)  \Big) \|\mathcal{A}(i\eta, \cdot)\|_{L^\infty( B_r ) } \|f\|_{L^1( B_r )},
\end{eqnarray*}
which, together with \eqref{1020-1}, implies 
\begin{eqnarray}\label{0954-2}
 \limsup_{\eta\rightarrow+\infty} \Big( \frac{1}{\eta} \ln |F(i\eta)| \Big)
 \leq \sup_{ x\in \mathrm{supp}(f)}  \varphi(x).
\end{eqnarray}

Next, fix an arbitrary $\varepsilon>0$. Write $\tau := \displaystyle\sup_{ x\in \mathrm{supp}(f)}  \varphi(x)$. 
Since supp$(f)\subset \{ \varphi(x) \leq \tau\} \cap B_r$, we get from the definition of $F$ in \eqref{FourierIntegralOperator-202601002} that for each $\eta \geq M$,
\begin{eqnarray*}
 |F(i\eta)| &=& \Big| \int_{\mathbb R^n} e^{\eta \varphi(x)} \mathcal{A}(i\eta, x) f(x) dx \Big|
 = \Big| \int_{ \{ \varphi(x) \leq \tau\} } e^{\eta \varphi(x) } \mathcal{A}(i\eta, x) f(x) dx \Big|
 \\
 &\geq& \Big| \int_{ \{ \tau - \varepsilon < \varphi(x) \leq \tau\} } e^{\eta \varphi(x) } \mathcal{A}(i\eta, x) f(x) dx \Big|
 - \int_{ \{ \varphi(x) \leq \tau - \varepsilon \} } e^{\eta \varphi(x) }  |\mathcal{A}(i\eta, x)|  \,  |f(x)| dx 
 \\
 &\geq& e^{\eta(\tau-\varepsilon)} \Big( e^{-\eta(\tau-\varepsilon)}
 \Big| 
 \int_{ \{ \tau - \varepsilon < \varphi(x) \leq \tau\} } e^{\eta \varphi(x) } \mathcal{A}(i\eta, x) f(x) dx
\Big|
- \|\mathcal{A}(i\eta, \cdot)\|_{L^\infty( B_r )} \|f\|_{L^1( B_r )} \Big).
\end{eqnarray*}
Meanwhile, by Lemma \ref{an-unique-continu-L1-weight}, we have
\begin{eqnarray*}
    \liminf_{\eta\rightarrow+\infty}  \left[
    e^{-\eta(\tau-\varepsilon)}
     \Big| 
    \int_{ \{ \tau - \varepsilon < \varphi(x) \leq \tau\} } e^{\eta \varphi(x) } \mathcal{A}(i\eta, x) f(x) dx
    \Big|  
    \right]
    = +\infty.
\end{eqnarray*}
Then, the last two estimates, along with \eqref{1020-1}, imply
\begin{eqnarray*}
 \liminf_{\eta\rightarrow+\infty} \Big( \frac{1}{\eta} \ln |F(i\eta)| \Big)
 \geq \liminf_{\eta\rightarrow+\infty} \Big( \frac{1}{\eta} \ln e^{\eta(\tau-\varepsilon)} \Big)
 = \tau-\varepsilon.
\end{eqnarray*}
Since \(\varepsilon>0\) is arbitrary, we obtain
\begin{eqnarray*}
 \liminf_{\eta\rightarrow+\infty}  \Big(  \frac{1}{\eta} \ln |F(i\eta)|  \Big)
 \geq \tau =  \sup_{ x\in \mathrm{supp}(f)}  \varphi(x), 
\end{eqnarray*}
which, together with \eqref{0954-2}, gives the desired conclusion. This completes the proof. 
\end{proof}

\section{Proof of Theorem \ref{TH:SE}}
\label{sec:proof-of-theorem-se}

This section is devoted to the proof of Theorem \ref{TH:SE}, which establishes the unique determination of the support of the potential from multi-frequency far-field measurements under plane-wave incidence. The proof proceeds in three steps. 
First, in Subsection \ref{meromorphic-ext-subsec} we investigate the meromorphic extension of the scattering solutions  to complex wavenumbers. Theorem \ref{th:sw} establishes the well-posedness of the forward scattering problem for positive real wavenumbers and complex wavenumbers with large imaginary parts, and the meromorphic extension of the solution operator to the complex plane except for a discrete set of exceptional points.
Second, in Subsection \ref{high-freq-subsec} we derive high-frequency asymptotic formulas for the scattering solutions and far-field patterns when the wavenumber approaches infinity along the imaginary axis. Theorem \ref{lem3.3} provides the asymptotic behavior of the total field and the far-field pattern in this regime.
Finally, in Subsection \ref{proof-thm-subsec} we combine these analytical tools to prove the uniqueness result under two different measurement configurations: fixed-angle far-field data and backscattering far-field data. The proof strategy relies on the Paley-Wiener-type theorem (Theorem \ref{thm-growth-support}) to extract support information from the high-frequency behavior of the far-field patterns.

\subsection{Meromorphic extension to complex wavenumbers}
\label{meromorphic-ext-subsec}

In this subsection, we extend the solutions of the scattering problem \eqref{eqs}-\eqref{eq:radiation} to the setting of complex wavenumbers $\kappa \in \mathbb C$. The main result is stated as follows. 

\begin{theorem}\label{th:sw}
    Let  $q\in L^\infty(\mathbb{R}^3)$ satisfy supp$(q)\subset B_R$ and $\operatorname{Im} q\ge 0$.     Let $d\in \s^2$. 
    Then, for each $\kappa >0$, the scattering problem \eqref{eqs}-\eqref{eq:radiation} admits a unique solution $u(\cdot; d, \kappa)$ in the space $ H^2_{\mathrm{loc}}(\R^3)\cap C^1\big( \mathbb R^3 \backslash B_R \big)$. 
    
    Furthermore, when regarded as a function of $\kappa$, the solution \(u(\cdot; d, \kappa)\) admits a unique meromorphic extension from \(\mathbb C\) to \(H^2_{\mathrm{loc}}(\R^3)\), with possible poles in a discrete set $\mathcal{I}$, independent of \(d \in \mathbb S^2\), satisfying that
    \begin{align}\label{DiscreteWaveNumbers-20261004}
        \mathcal{I}\cap  \big\{ z \in \mathbb R ~:~ z > 0 \big\}=\emptyset
        ~\text{ and }~
        \mathcal{I}\cap \big\{ z \in \mathbb C ~:~ \operatorname{Im} z > M \big\}=\emptyset
        \text{ for some }  M >0. 
    \end{align}
\end{theorem}

\begin{remark}
    From \eqref{DefineSetOfPoles-20261005} and \eqref{DefinitionOfCompactOperators-20261004} in the proof of Theorem \ref{th:sw}, the discrete set $\mathcal{I}$ is characterized by the following homogeneous integral equation with $\kappa \in \mathbb C$: 
    \begin{align}\label{eigenfunction-20261005}
        u(x) = \int_{B_R} 
        \frac{e^{i \kappa |x-y|}}{4\pi |x-y|} q(y) u(y) 
        dy, ~x\in B_R,  
    \end{align}
    which is a truncated version of the Lippmann--Schwinger equation \eqref{LS} (see also \eqref{fm}).  
    A complex number $\kappa \in \mathbb C$ belongs to $\mathcal{I}$ if and only if the eigenvalue problem \eqref{eigenfunction-20261005} has a non-trivial solution in $L^2(B_R)$. 
\end{remark}

As a consequence of Theorem \ref{th:sw}, the following corollary establishes both the integral expression and the meromorphic extension of the far-field patterns for complex wavenumbers. 

\begin{corollary}\label{Corollary:FarField-20261005}
Assume the hypotheses of Theorem \ref{th:sw} and use the same notation. Let $d \in \mathbb S^2$.  Write $u(\cdot; d, \kappa)$ for the meromorphic extension of the total field with $\kappa \in \mathbb C$, introduced in Theorem \ref{th:sw}. 
Then, the following function
\be\label{FP}
u^\infty(\hat{x}; d, \kappa) :=\frac{1}{4\pi}\int_{ B_R } e^{-i \kappa \hat{x}\cdot y}q(y) u(y; d, \kappa)\,dy, ~\hat{x}\in\s^2
\en
as a function of $\kappa$, is meromorphic from $\mathbb C$ to $C( \mathbb S^2)$, with possible poles only in $\mathcal{I}$, where $\mathcal{I}$ satisfies \eqref{DiscreteWaveNumbers-20261004}. Furthermore, for each $\kappa \in \mathbb C \backslash \mathcal{I}$,
\begin{align*}
   u(x; d, \kappa) - e^{i \kappa x \cdot d} 
    = \frac{e^{i\kappa |x|}}{|x|} \Big( u^\infty(\hat x;d,\kappa)+ \mathcal{O}\big(|x|^{-1}\big)
    \Big)
   ~\text{ as }  |x|\to+\infty,
\end{align*}
uniformly in all directions \(\hat x = x/|x|\). 
\end{corollary}

To prove Theorem \ref{th:sw}, we reformulate \eqref{eqs}-\eqref{eq:radiation} as the Lippmann-Schwinger integral equation, and then extend this integral equation to complex wavenumbers $\kappa \in \mathbb C$. 
For this purpose, we consider $\varPhi(\cdot; \kappa)$ with $\kappa \in \C$, which is a fundamental solution of \(-\Delta-\kappa^2\), that is, a solution of $(\Delta+\kappa^2)u= - \delta_0$ in $\R^3$, and satisfies the Sommerfeld radiation condition \eqref{eq:radiation}. More precisely, 
\be\label{fm}
\varPhi(x; \kappa) := \frac{e^{i \kappa |x|}}{4\pi |x|},~x\in \R^3\setminus\{0\}
~\text{ for each }
\kappa \in \mathbb C. 
\en
For \(d\in\mathbb S^2\) and \(\kappa\in\mathbb C\), the Lippmann-Schwinger integral equation reads
\be\label{LS}
u(x; d, \kappa) = e^{i \kappa x\cdot d}+\int_{B_R} 
\varPhi(x-y; \kappa) q(y) u(y; d, \kappa) 
dy, ~x\in \R^3.
\en
Any function $u\in L^2_{\mathrm{loc}}(\mathbb R^3)$ satisfying \eqref{LS}  is said to be a solution to this equation. 

It is well known that there is an equivalence between the forward scattering problem \eqref{eqs}-\eqref{eq:radiation} and the Lippmann-Schwinger integral equation.  A result of this kind was proved in \cite[Theorem 8.3, p. 309]{CK} for the Helmholtz equation with a piecewise continuous refractive index; here, we allow \(q\) to be merely bounded.

\begin{lemma}\label{thm:LS-equivalence}
Let $q\in L^\infty(\mathbb{R}^3)$ satisfy supp$(q)\subset B_R$. Let $d\in \mathbb S^2$ and $\kappa>0$. Then, the following two statements are true:
\begin{itemize}
    \item[(i)] If $u\in H^2_{\mathrm{loc}}(\mathbb{R}^3) \cap C^1\big( \mathbb R^3 \backslash B_R \big)$ is a solution of the scattering problem \eqref{eqs}-\eqref{eq:radiation}, then $u$ solves the Lippmann--Schwinger equation \eqref{LS}. 
    \item[(ii)] If $u\in L_{\mathrm{loc}}^2(\mathbb{R}^3)$ solves the Lippmann--Schwinger equation \eqref{LS}, then $u\in H^2_{\mathrm{loc}}(\mathbb{R}^3)\cap C^1\big( \mathbb R^3 \backslash B_R \big)$ and it is a solution of the scattering problem \eqref{eqs}-\eqref{eq:radiation}.
\end{itemize}
\end{lemma}

The proof of Lemma \ref{thm:LS-equivalence} is given in Appendix \ref{sec:proof-of-ls-equivalence}. We are now in a position to prove Theorem \ref{th:sw} and Corollary \ref{Corollary:FarField-20261005}. 

\begin{proof}[Proof of Theorem \ref{th:sw}]
The proof is classical and is based on the analytic Fredholm theorem. Observe that the Lippmann-Schwinger equation \eqref{LS} is equivalent to the following simplified version over $B_R$:
\begin{align}\label{LS-Alternative}
u(x; d, \kappa) = e^{i \kappa x\cdot d}+\int_{B_R} 
\varPhi(x-y; \kappa) q(y) u(y; d, \kappa) 
dy, ~x\in B_R.
\end{align}
To place this equation within the framework of the analytic Fredholm theorem, we define the family of operators $\{ K(\kappa)\}_{\kappa \in \mathbb C}$  in $\mathcal L\big( L^2(B_R) \big)$: for each $\kappa \in \mathbb C$, 
\begin{align}\label{DefinitionOfCompactOperators-20261004}
    \big( K(\kappa)  w \big) (x)  := \int_{B_R}  \varPhi(x-y;\kappa) q(y) w(y) dy,
    ~x\in B_R
    \text{ for each } w \in L^2(B_R). 
\end{align}
It is clear that the map $\kappa \mapsto K(\kappa)$ is analytic from $\mathbb C$ to $\mathcal L\big( L^2(B_R) \big)$. 
The remainder of the proof is organized into the following steps. 

\vskip 5pt
\noindent\textit{
Step 1. Show that for each $\kappa \in \mathbb C$, the operator $K(\kappa)$ is compact from $L^2(B_R)$ to $L^2(B_R)$.
}

Fix an arbitrary $\kappa \in \mathbb C$. Since the kernel of \(K(\kappa)\) belongs to \(L^2(B_R \times B_R)\), the Hilbert-Schmidt theorem (see e.g. \cite[Theorems VI.22-23, p. 210]{RS72}) implies that \(K(\kappa)\) is compact on $L^2(B_R)$.

By the way, we will show a regularity result of $K(\cdot)$ which has no relation with the conclusion in Step 1 but will be used in later proofs. This result is claimed as follows:
\begin{align}\label{ImprovedRegularity-20261004}
    K(\kappa) w  \in  H^2(B_R)
    \text{ for each }   w \in L^2(B_R).
\end{align}
Indeed, fix a $w \in L^2(B_R)$. By the definition of $K(\kappa)$ (in  \eqref{DefinitionOfCompactOperators-20261004}),  $K(\kappa) w$ is the restriction to $B_R$ of  the following function
\begin{align}\label{ExtensionOfw-20261004}
    u_w^{\mathrm{sc}}(x) := \int_{B_R}  \varPhi(x-y;\kappa) q(y) w(y) dy,
    ~x\in \mathbb R^3,
\end{align}
which satisfies the elliptic equation
\begin{align}\label{QuationForExtensionOfw-20261004}
    ( - \Delta - \kappa^2)  u_w^{\mathrm{sc}} = q \tilde{w}    \text{ in }  L^2(\mathbb R^3). 
\end{align}
Here, $\tilde w$ is the extension of \(w\) by zero outside \(B_R\), and \(\varPhi(\cdot ;\kappa)\) is a fundamental solution of \(-\Delta - \kappa^2\). By the ellipticity of $\Delta$, it is clear that
\begin{align*}
    u_w^{\mathrm{sc}} \in H_{\mathrm{loc}}^2(\mathbb R^3)  \text{ and }   
    K(\kappa) w = u_w^{\mathrm{sc}}|_{B_R} \in H^2(B_R). 
\end{align*}
Thus, the  claim \eqref{ImprovedRegularity-20261004} is true.

\vskip 5pt
\noindent\textit{
    Step 2. Prove that 
    \begin{align}\label{LimitAlongImaginaryAxis-20261004}
        \lim_{ \kappa \in \mathbb C, \, \operatorname{Im} \kappa \rightarrow +\infty}
        \| K(\kappa) \|_{ \mathcal L( L^2(B_R) ) }
        = 0 .
    \end{align}
}

Using \eqref{DefinitionOfCompactOperators-20261004} and \eqref{fm}, we obtain, for \(\kappa = \kappa_1 + i \kappa_2\) with \(\kappa_1,\kappa_2 \in \mathbb R\),
\begin{eqnarray*}
    \| K(\kappa) \|_{\mathcal L( L^2(B_R) )}
    &\leq&
    \frac{1}{4\pi} \left( \int_{B_R}   \int_{B_R}
    \left( \frac{ e^{-\kappa_2 |x-y| } }{|x-y|} |q(y)| \right)^2  dy dx
    \right)^{\frac{1}{2}}\\
    &\leq& \frac{ \sqrt{|B_R|} }{4\pi} \left(\int_{B_R} |q(y)|^2dy\right)^{1/2}\,\left(\int_{B_R}\int_{B_R}\frac{ e^{-\kappa_2 |x-y| } }{|x-y|^2}  dxdy\right)^{1/2}.
\end{eqnarray*}
Then, by the Lebesgue dominated convergence theorem, \eqref{LimitAlongImaginaryAxis-20261004} follows.

\vskip 5pt
\noindent\textit{
    Step 3. Prove that, for each $\kappa>0$, the operator $I-K(\kappa)$ is invertible on $L^2(B_R)$.
}

Fix an arbitrary \(\kappa>0\). By the compactness of $K(\kappa)$ over $L^2(B_R)$ (see Step 1), it suffices to show that the following equation
\begin{align}\label{KernelOfK-20261004}
    (I - K(\kappa)) w =0   \text{ in } L^2(B_R)
\end{align}
has only the trivial solution $w=0$ in $L^2(B_R)$. For this purpose, let $w\in  L^2(B_R)$ be a solution to the above equation. Let $u_w^{\mathrm{sc}}$ be given in \eqref{ExtensionOfw-20261004} with current $w$. It follows from \eqref{ExtensionOfw-20261004},  \eqref{KernelOfK-20261004} and \eqref{QuationForExtensionOfw-20261004} that $u_w^{\mathrm{sc}} |_{B_R} = w$ and that
\begin{align}\label{EllipticEquaitonOFuw-20261004}
    (-\Delta - \kappa^2) u_w^{\mathrm{sc}}  = q u_w^{\mathrm{sc}}  ~\text{ in } \mathbb R^3.
\end{align}
Next, multiplying \eqref{EllipticEquaitonOFuw-20261004} by \(\overline{u_w^{\mathrm{sc}}}\), integrating over \(B_r\) for \(r>0\), and applying Green's formula, we obtain
\ben
\int_{B_r} |\nabla u_w^{\mathrm{sc}}|^2-\kappa^2|u_w^{\mathrm{sc}}|^2-q(x)\,|u_w^{\mathrm{sc}}|^2\,dx-\int_{\partial B_r} \partial_\nu u_w^{\mathrm{sc}} \overline{u_w^{\mathrm{sc}}}\,ds=0,
\enn 
where $\nu$ is the outward unit normal on $\partial B_r$. Taking the imaginary part yields
\begin{equation}\label{rellich}
    \operatorname{Im}\left(\int_{\partial B_r} \partial_\nu u_w^{\mathrm{sc}} \overline{u_w^{\mathrm{sc}}}\,ds\right)+\int_{B_r}(\operatorname{Im} q)\, |u_w^{\mathrm{sc}}|^2\,dx=0,
    ~\forall\, r>0.
\end{equation}
Meanwhile, since $\kappa>0$, one can check from \eqref{ExtensionOfw-20261004} and \eqref{fm} (see also \eqref{AsymptoticsOfFundamentalSolution-20261004}) that
$$
\partial_{\nu}u_w^{\mathrm{sc}} \overline{u_w^{\mathrm{sc}}}
 = i\kappa |u_w^{\mathrm{sc}}|^2  + o(|x|^{-2})
 ~\text{ as } |x| \rightarrow +\infty. 
$$
Inserting this into \eqref{rellich} gives
$$
\kappa \int_{\partial B_r} |u_w^{\mathrm{sc}}|^2 \, dS+\int_{B_r} (\operatorname{Im} q) |u_w^{\mathrm{sc}}|^2 \, dx+o(1)=0
~\text{ as }  r \rightarrow +\infty.
$$
Since $\kappa>0$ and $\operatorname{Im}q\ge 0$, both integrals above are nonnegative. Therefore
$$
\frac{1}{r}\int_{r<|x|<2r} |u_w^{\mathrm{sc}}|^2\,dx
\to 0
~\text{ as }  r \rightarrow +\infty.
$$
Outside $B_R$, the potential $q$ vanishes, so it follows from \eqref{EllipticEquaitonOFuw-20261004} that 
$$
(-\Delta-\kappa^2)u_w^{\mathrm{sc}}=0 ~\text{ when } |x|>R.
$$
Applying Rellich's lemma (see e.g. \cite[Proposition 5.2, p. 5482]{RS20-2}), we obtain
$$
u_w^{\mathrm{sc}}=0 ~\text{ when } |x|>R.
$$
We now apply the unique continuation property for elliptic equations with complex-valued \(L^\infty\) potentials (see e.g. \cite[Theorem 14.12, p. 610]{V25}) to get that 
$u_w^{\mathrm{sc}} \equiv 0 ~\text{ in } \mathbb{R}^3$. 

Finally, since $u_w^{\mathrm{sc}} |_{B_R} = w$ (see the property above \eqref{EllipticEquaitonOFuw-20261004}), we conclude that $w=0$ in $L^2(B_R)$. Thus, equation \eqref{KernelOfK-20261004} has only the trivial solution. 

\vskip 5pt
\noindent\textit{
    Step 4. Proof of the desired conclusions.
}

Write $\mathbb R_+ := \big\{ x \in \mathbb R ~:~ x >0 \big\}$. From Steps 2--3, there exists a sufficiently large \(M>0\) such that the operator $I - K(\kappa)$ is invertible for each $\kappa$ in the following set
\begin{align*}
    \mathbb R_+  \cup  \big\{ z\in \mathbb C  ~:~  \operatorname{Im} z > M\}.
\end{align*}
Then, by  the analytic Fredholm theorem (see e.g., \cite[Theorem C.8, p. 514]{DZ19}) applied to $K(\cdot)$ (see Step 1), the following set
\begin{align}\label{DefineSetOfPoles-20261005}
    \mathcal{I} := \big\{ \kappa \in \mathbb C~:~  K(\kappa) \text{ has the eigenvalue } 1 \big\}
\end{align}
is discrete in $\mathbb C$ (and independent of $d \in \mathbb S^2$), and satisfies
\begin{align*}
    \mathcal{I} \cap \mathbb R_+   = \emptyset
    ~\text{ and }~
    \mathcal{I}  \cap  \big\{ z\in \mathbb C  ~:~  \operatorname{Im} z > M \big\} = \emptyset.
\end{align*}
Furthermore, \(I-K(\kappa)\) is invertible for \(\kappa \in \mathbb C\backslash \mathcal{I}\), and \((I-K(\cdot))^{-1}\) is meromorphic on \(\mathbb C\).  

Now, we deduce from the above, \eqref{DefinitionOfCompactOperators-20261004}, and \eqref{LS-Alternative} that the Lippmann-Schwinger equation \eqref{LS} has a unique solution in $L^2_{\mathrm{loc}}(\mathbb R^3)$ for each $\kappa \in \mathbb C \backslash \mathcal{I}$. Moreover, this solution is analytic in the variable $\kappa$. 

In particular, equation \eqref{LS} has a unique solution $u(\cdot;d,\kappa)$ in $L^2_{\mathrm{loc}}(\mathbb R^3)$ for each $\kappa>0$. Then, by Lemma \ref{thm:LS-equivalence}, this solution $u(\cdot;d,\kappa)$ is also the unique solution to the scattering problem \eqref{eqs}-\eqref{eq:radiation} with $\kappa>0$.

Furthermore, the solution \(u(\cdot;d,\kappa)\), as a function of $\kappa$, extends analytically to \(\mathbb C\backslash \mathcal{I}\), where it solves \eqref{LS} for complex wavenumbers. The desired conclusion follows.

\vskip 5pt
In summary, the proof is completed. 
\end{proof}

\begin{proof}[Proof of Corollary \ref{Corollary:FarField-20261005}]
The conclusions follow directly from Theorem \ref{th:sw}, \eqref{LS} and the asymptotic expansion \eqref{AsymptoticsOfFundamentalSolution-20261004} of $\varPhi(\cdot; \kappa)$ with $\kappa \in \mathbb C$. 
\end{proof}

\subsection{High-frequency asymptotics}
\label{high-freq-subsec}
 
 This subsection presents the leading asymptotic expansions of the meromorphically extended far-field patterns as the wavenumber tends to infinity along the imaginary axis. 
 
 We continue to denote the meromorphic extensions of the total field and the far-field pattern by $u(\cdot;d,\kappa)$ and $u^{\infty}(\cdot;d,\kappa)$ introduced in Theorem \ref{th:sw} and Corollary \ref{Corollary:FarField-20261005}, respectively. 

\begin{theorem}\label{lem3.3}
    Assume the same hypotheses of Theorem \ref{th:sw}. Let $d \in \mathbb S^2$.  Then the total field and the far-field pattern have the following leading asymptotic expansions as $\eta \rightarrow +\infty$: 
\begin{eqnarray}\label{asfa}
\begin{array}{r@{~}l}
u(x; d, i\eta) &= e^{- \eta x\cdot d} \big(1 + o(1) \big)
~\mbox{ in } L^\infty(B_R),
\vspace{0.5em}\\
u^\infty(\hat{x}; d,  i\eta) &= \displaystyle\frac{1}{4\pi} \displaystyle\int_{ B_R } e^{\eta(\hat{x} -d )\cdot y} q(y)  \big( 1 + o(1) \big) dy
~\mbox{ in } C(\mathbb S^2).
\end{array}
\end{eqnarray} 
Here, \(o(1)\) denotes a remainder that tends to zero in \(L^{\infty}(B_{R})\) as \(\eta\to+\infty\).
\end{theorem}

\begin{proof}
Let $M>0$ be given by Theorem \ref{th:sw}.  Fix an arbitrary $\eta> M$. Set
\begin{align}\label{DefinitionOfAmplitude-20261005}
    v(x; d, i\eta) := e^{ \eta x \cdot d} u(x; d, i \eta), ~ x\in B_R. 
\end{align}
Since $u(\cdot;d,i\eta)$ solves equation \eqref{LS}, the function $v$ satisfies
\begin{align}\label{IntegralEquationForAmplitude-20261005}
    v(x;d, i\eta ) = 1 + \frac{1}{4\pi} \int_{B_R}  \frac{ e^{ -\eta(|x-y|-d\cdot (x-y))} }{|x-y|} q(y)v(y;d,i\eta) dy,
    ~ x \in B_R.
\end{align}
The remainder of the proof is organized into the following three steps.

First, we claim 
\begin{align}\label{UniformBoundOfAmplitude-20261005}
    \limsup_{\tau\rightarrow+\infty} \| v(\cdot; d, i\tau) \|_{ L^{\infty}(B_R) }
    < +\infty.
\end{align}
Indeed, introduce the following operator $\widetilde{K}(\tau) $ in $\mathcal{L}( L^2(B_R))$ with $\tau>0$: for each $f \in L^2(B_R)$, 
\begin{align*}
    \big( \widetilde{K}(\tau)  f \big) (x) :=
    \frac{1}{4\pi} \int_{B_R}   \frac{ e^{ -\tau(|x-y|-d\cdot (x-y))} }{|x-y|} q(y) f(y) dy,
    ~ x \in B_R. 
\end{align*}
A direct calculation shows that the kernel of $\widetilde{K}(\tau)$ converges to $0$ in \(L^2(B_R \times B_R)\) as $\tau \rightarrow +\infty$, which implies
\begin{align*}
    \limsup_{\tau\rightarrow+\infty}  \| \widetilde{K}(\tau) \|_{ \mathcal{L}( L^2(B_R)) } =0.
\end{align*}
From this and \eqref{IntegralEquationForAmplitude-20261005}, we determine
\begin{align*}
    \limsup_{\tau\rightarrow+\infty} \| v(\cdot; d, i\tau) \|_{ L^{2}(B_R) }
    < +\infty.
\end{align*}
Meanwhile, when $\tau>0$,  it follows from \eqref{IntegralEquationForAmplitude-20261005} (with $\eta$ replaced by $\tau$) that
\begin{align*}
    \| v( \cdot ;d, i\tau ) \|_{ L^\infty(B_R ) }  \leq&  1 + \frac{1}{4\pi}
    \sup_{x\in B_R}  \int_{B_R} \frac{ |q(y)| }{|x-y|} |v(y;d,i\tau)| dy
    \nonumber\\
    \leq& 1 + \frac{1}{4\pi}   \Big( \int_{B_{2R}} \frac{dy}{|y|^2}  \Big)^{1/2}
    \|q\|_{ L^\infty(B_R)  }  \|v(\cdot;d,i\tau)\|_{ L^2(B_R)  }.
\end{align*}
The last two conclusions imply \eqref{UniformBoundOfAmplitude-20261005}. 

Second, for each $x\in B_R$ and $\delta>0$, the following set 
  \begin{eqnarray}\label{Omega-delta-epsilon}
   \Omega_\delta(x):= \big\{ y\in B_R:\;\;  |x-y| - d\cdot(x-y) \leq \delta \big\},
  \end{eqnarray}
  satisfies
   \begin{eqnarray}\label{EstimateOfIntegralOverSmallSet-20261005}
      \int_{\Omega_\delta(x)} |q(y)|^2 dy \leq  
      5\pi R^2\delta ||q||_{L^\infty(B_R)}^2.
  \end{eqnarray}
  Fix an $x\in B_R$ and a $\delta>0$. By coordinate translation and rotation, to estimate the volume of $\Omega_\delta(x)$, it suffices to compute the volume measure of the following domain
  \begin{eqnarray*}
   \Omega'_\delta(x):= \big\{ y\in B_{2R}\cap(x+B_R) ~:~  |y|+{\bf e}_3\cdot y\leq \delta \big\}
   \text{ with }  {\bf e}_3:=(0,0, 1).
  \end{eqnarray*}
 It is obvious that
 \begin{eqnarray*}
 \Omega'_{\delta} (x)&\subset&  \Big\{  y=(y_1,y_2,y_3) \in \R^3 ~:~  |y_3|<2R,~ \sqrt{y_1^2+y_2^2+y_3^2}+y_3 \leq \delta \Big\}
  \\
  &=& \big\{   
    y=(y_1,y_2,y_3) \in \R^3 ~:~ |y_3|<2R,~ y_1^2+y_2^2 \leq \delta^2 - 2 \delta y_3 
  \big\}
 \\
 &=&  \Big\{   
    y=(y_1,y_2,y_3) \in \R^3 ~:~   -2R< y_3 \leq \frac{\delta}{2},~ y_1^2+y_2^2 \leq \delta^2 - 2 \delta y_3
\Big\}.
\end{eqnarray*}
Thus, the measure of $\Omega'_{\delta}(x)$ is estimated as follows:
\begin{eqnarray*}
|\Omega'_{\delta}(x)| &\leq& \int_{-2R}^{ \frac{\delta}{2} } \left( \int_{y_1^2+y_2^2 \leq \delta^2 - 2 \delta y_3} d y_1 d y_2  \right) d y_3
= \int_{-2R}^{ \frac{\delta}{2} } \pi (\delta^2 - 2 \delta y_3 ) d y_3 \\
&=& \pi \delta (\delta^2/4+2R\delta+4R^2).
\end{eqnarray*}
Since $|\Omega'_\delta(x)|  \leq \frac{4}{3} \pi R^3$, the above estimate, together with the definitions of $\Omega_\delta(x)$ and $\Omega'_\delta(x)$, yields
\begin{eqnarray*}
  |\Omega_\delta(x)| \leq  |\Omega'_\delta(x)|  \leq \min\left\{ \pi \delta (\delta^2/4+2R\delta+4R^2),\frac{4}{3} \pi R^3 \right\}
 \leq   5\pi R^2\delta.
\end{eqnarray*} 
From this,  the estimate \eqref{EstimateOfIntegralOverSmallSet-20261005} follows at once. 

Third, applying the triangle and Hölder's inequalities to \eqref{IntegralEquationForAmplitude-20261005} and using \eqref{Omega-delta-epsilon}--\eqref{EstimateOfIntegralOverSmallSet-20261005}, we obtain, for \(x \in B_R\) and \(\delta>0\), 
  \begin{align*}
   &|v(x,d,i\eta)-1|   \leq   \int_{\Omega_\delta(x)} \frac{e^{-\eta[|x-y|-d\cdot(x-y)]}}{|x-y|} \, |q(y)| \, |v(y,d,i\eta)| dy
   \nonumber\\
   & \qquad
   + \int_{B_R\setminus\Omega_\delta(x)} \frac{e^{-\eta[|x-y|-d\cdot(x-y)]}}{|x-y|} \, |q(y)| \, |v(y,d,i\eta)| dy
   \nonumber\\
   \leq& \|v(\cdot,d,i\eta)\|_{L^\infty(B_R)} \Big(\int_{B_R} \frac{1}{|x-y|^2}  dy \Big)^{\frac{1}{2}} \left(\Big(\int_{\Omega_\delta(x)} |q(y)|^2 dy\Big)^{\frac{1}{2}}
   + e^{-\eta \delta} \Big(\int_{B_R\setminus\Omega_\delta(x)} |q(y)|^2  dy
   \Big)^{\frac{1}{2}}\right)\\ \no   
   \leq& \|v(\cdot,d,i\eta)\|_{L^\infty(B_R)} \Big(\int_{B_R} \frac{1}{|x-y|^2}  dy \Big)^{\frac{1}{2}} \Big((5\pi R^2\delta 
   )^{\frac{1}{2}}+ e^{-\eta \delta} (4\pi R^3/3)^{\frac{1}{2}}   \Big)||q||_{L^\infty(B_R)}.
  \end{align*}
  Since $\eta \geq M $ was arbitrary, the preceding estimate and \eqref{UniformBoundOfAmplitude-20261005} imply that, for some \(C>0\),
  \begin{eqnarray*}
   \limsup_{\eta \rightarrow +\infty}  \|v(\cdot; d,i\eta)-1\|_{L^{\infty}(B_R)} \leq C \delta^{\frac{1}{2}}
   ~\text{ for each }  \delta >0.
  \end{eqnarray*}
  This  implies
  \begin{align}\label{ConvergenceOfAmplitude-20261005}
       \lim_{\eta \rightarrow +\infty}  v(\cdot; d,i\eta) = 1
       \text{ strongly in }  L^\infty(B_R).
  \end{align}
  
  Finally, the first equality in \eqref{asfa} follows from \eqref{DefinitionOfAmplitude-20261005}  and \eqref{ConvergenceOfAmplitude-20261005}. The second one in \eqref{asfa} follows from \eqref{FP} and the first equality in \eqref{asfa}. This completes the proof. 
\end{proof}

\subsection{Proof of Theorem \ref{TH:SE}}
\label{proof-thm-subsec}

Take two potentials $q_1$ and $q_2$, both of which satisfy $\mathrm{Assumption}$ $\mathrm{\mathbf{(H)}}$. There are two convex domains $D_1,D_2$ in $\mathbb R^3$ such that 
\begin{align*}
    \overline{D_j} ={\rm supp} (q_j)  \subset B_R,
    ~j=1,2.
\end{align*}
For each $j \in \{1,2\}$,  denote by \(u^\infty_j\) the far-field pattern generated by plane waves incident on the inhomogeneity \((q_j, \overline{D_j})\). We shall prove \(\overline{D_1} = \overline{D_2}\) separately for the two measurement configurations, assuming in each case that the corresponding data for \(q_1\) and \(q_2\) coincide. 

(i) We first consider fixed-angle scattering data. Suppose that the scattering data for \(q_1\) and \(q_2\) coincide as follows:
\be\label{identity-s}
u^\infty_1(\theta; \pm d_0, \kappa)=u^\infty_2(\theta; \pm d_0, \kappa)
\quad\mbox{for all}~ \theta\in \s^2
\text{ and }  
\kappa\in(\kappa_{\min}, \kappa_{\max}).
\en
Then, by the meromorphic extension in Corollary \ref{Corollary:FarField-20261005}, there is a large $M>0$ such that
\begin{align*}
    u^\infty_1(\theta; \pm d_0, i \eta)=u^\infty_2(\theta; \pm d_0, i \eta)
    \quad\mbox{for all}~ \theta\in \s^2
    \text{ and }  
    \eta > M.
\end{align*}
Applying the asymptotic expansion \eqref{asfa} to both potentials and using the preceding equality, we obtain, for each \(\theta \in \mathbb S^2\) and \(d= \pm d_0\),  
\be\label{Feta} 
\int_{B_R} e^{\eta(\theta - d) \cdot x} q_1(x) (1 + o_1(1)) dx
=  \int_{B_R} e^{\eta(\theta - d) \cdot x} q_2(x) (1 + o_2(1)) dx,
~\eta > M,
\en
here the two terms $o_1(1)$ and $o_2(1)$ denote some infinitesimal functions in $L^{\infty}(B_R)$ with respect to $\eta \rightarrow +\infty$. 

     Next, we combine \eqref{Feta} with the Paley-Wiener-type theorem (Theorem \ref{thm-growth-support}) to extract information about the supports of \(q_1\) and \(q_2\). For this purpose, fix
     \begin{align}\label{ArbitraryDirections-20261003}
         \theta \in \mathbb S^2  \text{ and }   d=\pm d_0
         \text{ with }   \theta \neq d.  
     \end{align}
     Since each \(q_j\), \(j=1,2\), has a fixed sign near the boundary of its support (see $\mathrm{Assumption}$ $\mathrm{\mathbf{(H)}}$), we can apply Theorem \ref{thm-growth-support} to each potential, where
     \begin{align*}
         r=R,  ~\varphi(x) = (\theta - d) \cdot x,  ~
         \mathcal A(i\eta,x) = 1 + o_j(1), \text{ and }
         f=q_j,
     \end{align*}
     and use \eqref{Feta}  to find
     \begin{align*}
     \sup_{x\in \text{supp}(q_1)} \varphi(x)
     = \sup_{x\in \text{supp}(q_2)} \varphi(x),
     ~\text{i.e.,}~
     \sup_{x\in \overline{D_1} }    (\theta - d) \cdot x =
     \sup_{x\in \overline{D_2} }      (\theta - d) \cdot x.
     \end{align*}
     This, together with \eqref{ArbitraryDirections-20261003}, indicates that for each   $\theta \in \mathbb S^2$ and $d=\pm d_0$ with $\theta \neq d$, 
    \begin{align}\label{sup}
        \sup_{x\in D_1}   x \cdot (\theta - d) =
        \sup_{x\in D_2}     x \cdot (\theta - d).
    \end{align}
    The unit vectors \((\theta-d)/|\theta-d|\), with \(\theta \in \mathbb S^2\), \(d=\pm d_0\), and \(\theta\neq d\), cover all of \(\mathbb S^2\) except the directions orthogonal to \(d_0\). Thus, by density,
    \begin{align}\label{AllDirections-20261003}
        \sup_{x\in D_1}   x \cdot v =
        \sup_{x\in D_2}     x \cdot v
        ~\text{ for each }  v \in \mathbb S^2 .
    \end{align}
    
    Finally, because both $D_1$ and $D_2$ are convex, one can directly deduce from \eqref{AllDirections-20261003} that
    \begin{align*}
        \overline{D_1}=\overline{D_2}.
    \end{align*}
To see this, we argue by contradiction. Without loss of generality, suppose that there exists \(x_0 \in \mathbb R^3\) such that
    \begin{align*}
         x_0 \in \overline{D_1} \setminus \overline{D_2}.
    \end{align*}
    Then, since $\overline{D_2}$ is convex and closed, by the Hahn-Banach separation theorem, there is a $v_0 \in \mathbb S^2$ such that
    \begin{align*}
        \sup_{x \in D_1} x \cdot v_0 \geq x_0 \cdot v_0 > \sup_{x \in D_2} x \cdot v_0,
    \end{align*}
    which contradicts \eqref{AllDirections-20261003}. Therefore, \(\overline{D_1}=\overline{D_2}\) under the condition \eqref{identity-s}.

\vskip 5pt
(ii) We next consider backscattering data. Suppose that the scattering data for \(q_1\) and \(q_2\) coincide as follows:
\be\label{identity-s2}
u^\infty_1(\theta; -\theta, \kappa)=u^\infty_2(\theta; -\theta, \kappa) 
~~\mbox{for all}~ \theta\in \s^2
\text{ and }
\kappa\in(\kappa_{\min}, \kappa_{\max}).
\en
Then, by the meromorphic extension in Corollary \ref{Corollary:FarField-20261005}, there is a large $M>0$ such that
\begin{align*}
    u^\infty_1(\theta; -\theta, i \eta)=u^\infty_2(\theta; -\theta, i \eta)
    \quad\mbox{for all}~ \theta\in \s^2
    \text{ and }  
    \eta > M.
\end{align*}
By using the asymptotic expansion \eqref{asfa} twice, we determine from the last equality that for each $\theta \in \mathbb S^2$,
\be\label{NewFeta-20261003} 
\int_{B_R} e^{ \eta (2\theta \cdot x) } q_1(x) (1 + o_1(1)) dx
=  \int_{B_R} e^{ \eta (2\theta\cdot x) } q_2(x) (1 + o_2(1)) dx,
~\eta > M,
\en
where \(o_{1}(1)\) and \(o_{2}(1)\) denote remainders that tend to zero in \(L^{\infty}(B_{R})\) as \(\eta\to+\infty\).

Next, by an argument similar to the proof of \eqref{sup}, we apply the Paley-Wiener-type theorem (Theorem \ref{thm-growth-support}) twice and use \eqref{NewFeta-20261003} to obtain
\begin{align*}
    \sup_{x\in D_1}   x \cdot \theta =
    \sup_{x\in D_2}     x \cdot \theta
    ~\text{ for each }  \theta \in \mathbb S^2 .
\end{align*}
Because both $D_1$ and $D_2$ are convex, one can directly deduce from the last equality that $\overline{D_1}=\overline{D_2}$ (under the condition \eqref{identity-s2}). 

In summary, the proof of Theorem \ref{TH:SE} is completed.

\section{Proof of Theorem \ref{TH:near}}
\label{sec:proof of theorem near field}

This section is devoted to the proof of Theorem \ref{TH:near}, following an approach similar to that  in the proof of Theorem \ref{TH:SE} (see Section \ref{sec:proof-of-theorem-se}). To this end, we first establish the meromorphic extension and an asymptotic expansion of the total field for complex wavenumbers (Subsection \ref{sec:ps-extension-expansion}). We then present a geometric result relating a convex set to its minimum distances from points on a surrounding sphere (Subsection \ref{subsec:Support-MinimalDistance}). Finally, we prove Theorem \ref{TH:near} in Subsection \ref{sec:proof of theorem near field}. 

\subsection{Meromorphic extensions and high-frequency asymptotics}
\label{sec:ps-extension-expansion}

In this subsection, we extend the solutions of the scattering problem (1.6) to the setting of complex wavenumbers $\kappa \in \mathbb C$, and study both the meromorphic extension and the asymptotic expansion along the imaginary axis. The main results in this subsection are stated as follows. 

\begin{theorem}\label{th:ps}
    Let  $q\in L^\infty(\mathbb{R}^3)$ satisfy supp$(q)\subset B_R$ and $\operatorname{Im} q\ge 0$.  Let  $\mathcal{I}$  be the discrete set defined in \eqref{DefineSetOfPoles-20261005}, which satisfies \eqref{DiscreteWaveNumbers-20261004}.   Let $y \in \mathbb R^3$. 
    Then, the following statements are true:
    \begin{itemize}
        \item[(i)] For each $\kappa >0$, the point-source scattering problem 
        \eqref{psiw} admits a unique solution $u(\cdot; y, \kappa)$ in the space $ H^2_{\mathrm{loc}}(\R^3 \setminus \{y\} )\cap C^1\big( \mathbb R^3 \setminus ( B_R \cup \{y\}) \big)$. 
        
        Furthermore, as a function of \(\kappa\), the solution \(u(\cdot; y, \kappa)\) admits a unique meromorphic extension from \(\mathbb C\) to \(L^2_{\mathrm{loc}}(\R^3)\), with possible poles in  $\mathcal{I}$. 
         \item[(ii)]  The scattered field $u^{\mathrm{sc}}(\cdot;y,\kappa)$ admits a unique meromorphic extension from $\mathbb C$ to $C(\mathbb R^3)$, with possible poles only in \(\mathcal{I}\), given by
         \be\label{SF-20261005}
         u^{\mathrm{sc}}(x; y, \kappa) = \frac{1}{4\pi} \displaystyle\int_{B_R} \frac{e^{  i \kappa  |x-z|}}{|x-z|}
         q(z)u(z;y,\kappa) dz,
         ~ x\in\mathbb{R}^3.
         \en
        \item[(iii)] The total and scattered  fields  have the leading asymptotic expansions along the imaginary axis: when $\eta \rightarrow +\infty$, 
        \begin{eqnarray}\label{ps-asym}
            \begin{array}{r@{~}l}
                u(x; y, i\eta) &= \dfrac{e^{-\eta|x-y|}}{4\pi|x-y|} \big(1 + o(1) \big)
                ~\mbox{ in } L^\infty(B_R), 
                \vspace{0.5em}\\
                u^{\mathrm{sc}}(y; y, i\eta) &= \dfrac{1}{16\pi^2} \displaystyle\int_{B_R} \frac{e^{ - 2 \eta  |y-z|}}{|y-z|^2} q(z)  \big( 1 + o(1) \big)dz.
            \end{array}
        \end{eqnarray} 
Here, \(o(1)\) denotes a remainder that tends to zero in \(L^{\infty}(B_{R})\) as \(\eta\to+\infty\).
    \end{itemize}
\end{theorem}

\begin{remark}
    The set $\mathcal{I}$ in Theorem \ref{th:ps} is the same as that in Theorem \ref{th:sw} because both results use the same analytic family of compact operators (see \eqref{DefineSetOfPoles-20261005}). This set is characterized by the eigenvalue problem introduced in \eqref{eigenfunction-20261005}.
\end{remark}

To prove Theorem \ref{th:ps}, we transform \eqref{psiw} to the Lippmann-Schwinger integral equation, and then extend this integral equation to the setting of complex wavenumbers $\kappa \in \mathbb C$. 
For this purpose, let $\varPhi(\cdot;\kappa)$ be given by \eqref{fm}.  
The Lippmann--Schwinger equation with $y \in \mathbb R^3$ and $\kappa \in \mathbb C$ reads as: 
\begin{equation}\label{eq:LS}
    u(x;y,\kappa)=\varPhi(x-y; \kappa)
    +\int_{ B_R } \varPhi(x-z; \kappa)q(z)u(z;y,\kappa)\,\mathrm{d}z,
    ~ x\in\mathbb{R}^3\setminus\{y\}.
\end{equation}
Any function $u\in L^2_{\mathrm{loc}}(\mathbb R^3)$ satisfying \eqref{eq:LS}  is said to be a solution to this equation. 

The following lemma establishes the equivalence between the scattering problem \eqref{psiw} and the Lippmann-Schwinger equation \eqref{eq:LS}, its proof is quite similar to that of Lemma \ref{thm:LS-equivalence} and is omitted here. 
\begin{lemma}
    \label{thm:equivalence-ps}
    Let $q\in L^\infty(\mathbb{R}^3)$ satisfy supp$(q)\subset B_R$. Let $y\in \mathbb R^3$ and $\kappa>0$. Then, the following two statements are true:
    \begin{itemize}
        \item[(i)] If $u\in H^2_{\mathrm{loc}}(\R^3 \setminus \{y\} )\cap C^1\big( \mathbb R^3 \setminus ( B_R \cup \{y\}) \big)$ is a solution of the scattering problem \eqref{psiw}, then $u$ solves the Lippmann--Schwinger equation \eqref{eq:LS}. 
        \item[(ii)] If $u\in L_{\mathrm{loc}}^2(\mathbb{R}^3)$ solves the Lippmann--Schwinger equation \eqref{eq:LS}, then $u\in H^2_{\mathrm{loc}}(\R^3 \setminus \{y\} )\cap C^1\big( \mathbb R^3 \setminus ( B_R \cup \{y\}) \big)$ and it is a solution of the scattering problem \eqref{psiw}.
    \end{itemize}
\end{lemma}

We are now in a position to prove Theorem \ref{th:ps}.
\begin{proof}[Proof of Theorem \ref{th:ps}]
We prove statements (i)--(iii) in turn.

\vskip 5pt
(i) We adopt the ideas from the proof of Theorem \ref{th:sw} and reuse  partial arguments there. Let \(K(\kappa)\), \(\kappa \in \mathbb C\), be the operator defined in \eqref{DefinitionOfCompactOperators-20261004}. As shown in the proof of Theorem \ref{th:sw}, the map \(\kappa \mapsto K(\kappa)\) is analytic from \(\mathbb C\) to \(\mathcal L( L^2(B_R) )\) and takes values in the compact operators on \(L^2(B_R)\). Furthermore, $I - K(\kappa)$ is invertible for each $\kappa \in \mathbb C \backslash \mathcal{I}$. 

Now, observe from \eqref{fm}  that $\varPhi(\cdot-y; \kappa) \in L^2_{\mathrm{loc}}(\mathbb R^3)$ for each $\kappa \in \mathbb C$. Then, we deduce from the above argument and \eqref{DefinitionOfCompactOperators-20261004} that the Lippmann-Schwinger equation \eqref{eq:LS} has a unique solution in $L^2_{\mathrm{loc}}(\mathbb R^3)$ for each $\kappa \in \mathbb C \backslash \mathcal{I}$. Moreover, this solution is analytic in the variable $\kappa$. 

In particular, equation \eqref{eq:LS} has a unique solution $u(\cdot;y,\kappa)$ in $L^2_{\mathrm{loc}}(\mathbb R^3)$ for each $\kappa>0$. Then, by Lemma \ref{thm:equivalence-ps}, this solution $u(\cdot;y,\kappa)$ is also the unique solution to the scattering problem \eqref{psiw} with $\kappa>0$.

Furthermore, the solution \(u(\cdot;y,\kappa)\), as a function of $\kappa$, extends analytically to \(\mathbb C\backslash \mathcal{I}\), where it solves \eqref{eq:LS} for complex wavenumbers. The desired conclusions follow.

\vskip 5pt
(ii) Statement (ii) follows directly from part (i), Lemma \ref{thm:equivalence-ps}, and \eqref{eq:LS}.

\vskip 5pt
(iii) We adopt the ideas from the proof of Theorem \ref{lem3.3}. Fix an arbitrary sufficiently large \(\eta> 0\). Set
\begin{align}\label{ps:DefinitionOfAmplitude-20261005}
    v(x; y, i\eta) := u(x; y, i \eta)  \Big( 4\pi|x-y| e^{  \eta |x-y|} \Big), ~ x\in B_R . 
\end{align}
Introduce the following operator $\widehat{K}(\eta) $ in $\mathcal{L}( L^2(B_R))$: for each $f \in L^2(B_R)$, 
\begin{align}\label{ps:NewIntegralOperator-20261005}
    \big( \widehat{K}(\eta)  f \big) (x) :=
    \frac{|x-y|}{4\pi}
    \int_{ B_R }
    \frac{e^{-\eta[|x-z|+|y-z|-|x-y|]}}{|x-z|\,|y-z|}
    q(z) f(z) dz,
    ~ x \in B_R. 
\end{align}
Since $u(\cdot;y,i\eta)$ solves equation \eqref{eq:LS}, the function $v$ satisfies
\begin{align}\label{ps:IntegralEquationForAmplitude-20261005}
    v(x;y, i\eta ) = 1 + \Big( \widehat{K}(\eta) v(\cdot;y,i\eta) \Big)(x) ,
    ~ x \in B_R.
\end{align}
A direct calculation shows that the kernel of \(\widehat{K}(\tau)\) converges to $0$ in \(L^2(B_R \times B_R)\) as \(\tau \rightarrow +\infty\), which implies
\begin{align*}
    \limsup_{\tau\rightarrow+\infty}  \| \widehat{K}(\tau) \|_{ \mathcal{L}( L^2(B_R)) } =0.
\end{align*}
From this and \eqref{ps:IntegralEquationForAmplitude-20261005}, we determine
\begin{align}\label{ps:UniformL2Bound-20261005}
    \lim_{\tau\rightarrow+\infty}v(\cdot; y, i\tau) =1
    \text{ in }  L^{2}(B_R).
\end{align}

Now, for each $x\in B_R$ and $\delta>0$, the volume of the intersection with \(B_R\) of the ellipsoid with major semiaxis \(\frac{|x-y|+\delta}{2}\) and focal radius \(\frac{|x-y|}{2}\)
\begin{eqnarray}\label{ps:Omega-delta-epsilon}
    \Omega_\delta(x,y):= \big\{ z\in B_R ~:~ |x-z|+|y-z|-|x-y| \leq \delta \big\}
\end{eqnarray}
is bounded by $ \frac{\pi}{3} (|x-y|+\delta)^2 \delta $. 
We deduce from \eqref{ps:IntegralEquationForAmplitude-20261005} and \eqref{ps:NewIntegralOperator-20261005} that when $ x \in B_R$ and $\delta>0$, 
\begin{align*}
       | v(x;y,i\eta) -1| \leq&
    \frac{ R + |y| }{4\pi}  \|q\|_{L^\infty(B_R)}  \| v(\cdot; y, i\eta) \|_{ L^{\infty}(B_R) }   \times 
    \nonumber\\
    &   \bigg(
    \int_{\Omega_\delta(x,y)}      \frac{1}{|x-z|\,|y-z|}     dz
    + e^{-\eta \delta}   \int_{B_R \setminus \Omega_\delta(x,y)} 
         \frac{1}{|x-z|\,|y-z|}     dz
    \bigg)
    \nonumber\\
    \leq&
    \frac{ R + |y| }{4\pi}  \|q\|_{L^\infty(B_R)}  \| v(\cdot; y, i\eta) \|_{ L^{\infty}(B_R) }   \times 
    \nonumber\\
    &   \bigg(
    \int_{ (x - \Omega_\delta(x,y) ) \cup (y - \Omega_\delta(x,y) )}      \frac{1}{|z|^2}     dz
    + e^{-\eta \delta}   \int_{B_{2R+|y|} } 
    \frac{1}{|z|^2}     dz
    \bigg)
\end{align*}
where Hölder's inequality is applied twice in the last step. 
Meanwhile, by using \eqref{ps:IntegralEquationForAmplitude-20261005},  \eqref{ps:NewIntegralOperator-20261005}, and \eqref{ps:UniformL2Bound-20261005}, we obtain
\begin{align*}
    \limsup_{\tau\rightarrow+\infty} \| v(\cdot; y, i\tau) \|_{ L^{\infty}(B_R) }
    < +\infty.
\end{align*}
Thus, by the absolute continuity of the Lebesgue integral,  the last two inequalities yield
\begin{align*}
    \lim_{\tau\rightarrow+\infty}v(\cdot; y, i\tau) =1
    \text{ in }  L^{\infty}(B_R).
\end{align*}
This, along with \eqref{ps:DefinitionOfAmplitude-20261005}, gives the first equality in \eqref{ps-asym}. The second one in \eqref{ps-asym} follows from the first one and \eqref{SF-20261005}. This completes the proof.     
\end{proof}

\subsection{Support and minimal distances}
\label{subsec:Support-MinimalDistance}

This subsection presents a geometric result that characterizes a convex set by its minimum distances from points on a surrounding sphere.

\begin{lemma}\label{convex-equal}
    Let $R>0$ and let $D_1,D_2$ be two convex domains in $\mathbb R^n$ satisfying
   \begin{align}\label{lem:ContainedInBall-20261004}
       \overline{D_1}  \cup  \overline{D_2}  \subset B_{\mathbb R^n}(0,R),
   \end{align}
   where $B_{\mathbb R^n}(0,R)$ is the open ball centered at \(0\) with radius \(R\). 
  Suppose that
  \begin{align}\label{lem:TestingData-20261004}
      \min_{x \in \overline{D_1}} |x - y| = \min_{x \in \overline{D_2}} |x - y|
      \text{ when }  |y| =R.
  \end{align}
  Then \(\overline{D_1} = \overline{D_2}\). 
\end{lemma}

\begin{proof}
We prove the statement by contradiction. Assume that $\overline{D_1} \neq \overline{D_2}$. Without loss of generality, suppose $\overline{D_1} \backslash \overline{D_2} \neq \emptyset$;  otherwise, we swap the roles of $\overline{D_1}$ and $\overline{D_2}$. Choose a point 
\begin{align}\label{SeparationPoint-20261004}
    z_0 \in \overline{D_1} \backslash \overline{D_2}.
\end{align}
Since \(\overline{D_2}\) is closed and convex, \eqref{SeparationPoint-20261004} and the Hahn-Banach separation theorem yield a unit vector \(\mathbf{n} \in \mathbb{R}^n\)  and a real number $\alpha$ such that
\begin{align}\label{eq:sep2}
    \sup_{ x \in \overline{D_1} }   x \cdot  \mathbf{n}
   \geq  z_0 \cdot  \mathbf{n} > \alpha \geq
\sup_{ x \in \overline{D_2} }   x \cdot  \mathbf{n}.
\end{align}
This means that $\overline{D_2}$ lies entirely in the half-space $\{ z  ~:~  z \cdot \mathbf{n}  \leq \alpha \}$, while $z_0$ is strictly on the opposite side. 

We now construct a point \(y \in \mathbb R^n\) with \(|y|=R\) for which \eqref{lem:TestingData-20261004} does not hold. For this purpose, consider the ray emanating from $z_0$ in the direction $\mathbf{n}$:
\begin{align}\label{RayFromBasePointAlongSeparationDirection-20261004}
    y(t) := z_0 + t \mathbf{n}, ~ t \geq 0.
\end{align}
Since $z_0 \in \overline{D_1} \subset B_{\mathbb R^n}(0,R)$ (see \eqref{lem:ContainedInBall-20261004}), there is a unique number $t_0 > 0$ such that
\begin{align*}
    R= | y(t_0) |  = | z_0 + t_0 \mathbf{n}|.
\end{align*}
This implies
\begin{align}\label{LowerBoundOfD1-20261004}
    \min_{x \in \overline{D_1}} |x - y(t_0)|  \leq  |z_0 - y(t_0)|
    = t_0. 
\end{align}
On the other hand, \eqref{RayFromBasePointAlongSeparationDirection-20261004} implies that, for each \(x \in \overline{D_2}\), 
\begin{align*}
   |x - y(t_0)| \geq  ( y(t_0) -x ) \cdot \mathbf{n}
   = (t_0 + z_0 \cdot \mathbf{n} ) - x \cdot \mathbf{n},
\end{align*}
which, together with \eqref{eq:sep2}, yields
\begin{align*}
    \min_{x \in \overline{D_2}} |x - y(t_0)|  > t_0.
\end{align*}
This, along with \eqref{LowerBoundOfD1-20261004}, indicates
\begin{align*}
    \min_{x \in \overline{D_2}} |x - y(t_0)|  >  \min_{x \in \overline{D_1}} |x - y(t_0)|,
\end{align*}
which contradicts \eqref{lem:TestingData-20261004} with $y=y(t_0)$. Therefore, \(\overline{D_1} = \overline{D_2}\), which completes the proof. 
\end{proof}

\subsection{Proof of Theorem \ref{TH:near}}
\label{sec:proof of theorem near field}

Take two potentials $q_1$ and $q_2$, both of which satisfy $\mathrm{Assumption}$ $\mathrm{\mathbf{(H)}}$. There are two convex domains $D_1,D_2$ in $\mathbb R^3$ such that 
\begin{align}\label{Supports-20261004}
    \overline{D_j} ={\rm supp} (q_j)  \subset B_R,
    ~j=1,2.
\end{align}
For each $j \in \{1,2\}$,  denote by \(u^{\mathrm{sc}}_j\) the scattered field generated by point-source incidence on the inhomogeneity \((q_j, \overline{D_j})\). 

 We aim to verify $\overline{D_1} = \overline{D_2}$ from identical point-source scattering data of $q_1,q_2$.  Suppose that the two potentials $q_1,q_2$ share the 
following same  scattering data:  
\be\label{identity-PointSource}
u^{\mathrm{sc}}_1(y; y, \kappa)=u^{\mathrm{sc}}_2(y; y, \kappa)
\quad\mbox{when}~ |y| = R
\text{ and }  
\kappa\in(\kappa_{\min}, \kappa_{\max}).
\en
By (ii) of Theorem \ref{th:ps}, there is a large $M>0$ such that
\begin{align*}
    u^{\mathrm{sc}}_1(y; y, i \eta)=u^{\mathrm{sc}}_2(y; y, i \eta)
    \quad\mbox{when}~ |y| = R
    \text{ and }  
    \eta > M.
\end{align*}
Applying the asymptotic expansion in (iii) of Theorem \ref{th:ps} to both potentials and using the preceding equality, we obtain, for \(|y| = R\),
\be\label{Feta-PointSource} 
\int_{B_R} e^{ - 2\eta |x-y| } \frac{q_1(x)}{|x-y|^2}  (1 + o_1(1)) dx
=  \int_{B_R} e^{ - 2\eta |x-y| } \frac{q_2(x)}{|x-y|^2}  (1 + o_2(1)) dx,
~\eta > M,
\en
where \(o_{1}(1)\) and \(o_{2}(1)\) denote remainders that tend to zero in \(L^{\infty}(B_{R})\) as \(\eta\to+\infty\).

Next, we combine \eqref{Feta-PointSource} with the Paley-Wiener-type theorem (Theorem \ref{thm-growth-support}) to extract information about the supports of \(q_1\) and \(q_2\). For this purpose, fix a point \(y \in \mathbb R^3\) satisfying
\begin{align}\label{ArbitraryPoint-20261004}
    |y| = R.  
\end{align}
By \eqref{Supports-20261004}, there is a small number $\varepsilon_0 \in (0,R)$ such that
\begin{align*}
    \text{supp}(q_1) \cup \text{supp}(q_2)   \subset  B_{R-\varepsilon_0}.
\end{align*}
Since each \(q_j\), \(j=1,2\), has a fixed sign near the boundary of its support (see $\mathrm{Assumption}$ $\mathrm{\mathbf{(H)}}$), we can apply Theorem \ref{thm-growth-support} to each potential, where
\begin{align*}
    r=R-\varepsilon_0,  ~\varphi(x) = -2|x-y|,  ~
    \mathcal A(i\eta,x) = 1 + o_j(1), \text{ and }
    f=\frac{q_j(x)}{|x-y|^2},
\end{align*}
and use \eqref{Feta-PointSource}  to find
\begin{align*}
    \sup_{x\in \text{supp}(q_1)} \varphi(x)
    = \sup_{x\in \text{supp}(q_2)} \varphi(x),
    ~\text{i.e.,}~
    \inf_{x\in \overline{D_1} }   |x - y| =
    \inf_{x\in \overline{D_2} }    |x - y|.
\end{align*}
This, together with \eqref{ArbitraryPoint-20261004}, indicates 
\begin{align}\label{sup-PointSource}
    \min_{x\in \overline{D_1} }   |x - y| =
    \min_{x\in \overline{D_2} }    |x - y|
    \text{ when }
    |y| =R.
\end{align}

Finally, it follows from Lemma \ref{convex-equal} and \eqref{sup-PointSource} that $\overline{D_1}=\overline{D_2}$ (under the condition \eqref{identity-PointSource}). The proof is completed. 

\section{Concluding remarks and future work}
\label{sec:Concluding remarks}

This paper has developed a novel approach to proving uniqueness for inverse scattering problems for the time-harmonic Schr\"odinger equation. Our methodology integrates two key components: 
(i) a Paley-Wiener-type theorem (Theorem \ref{thm-growth-support}) characterizing the support of \(q\) through the asymptotic growth of the Fourier transform of \(qu\) along imaginary frequencies; and (ii) the use of multi-frequency measurement data together with their analytic dependence on frequency. Within this framework, we have established two uniqueness results for the support \(\overline{D} = \operatorname{supp}(q)\): one using fixed-angle or backscattering far-field data under plane-wave incidence, and the other using near-field backscattering data under point-source incidence. Moreover, the potentials considered in this paper may be complex-valued, a case that has rarely been investigated in previous work.

A possible extension is to establish a more general Paley-Wiener-type theorem without the sign-definiteness condition near the support boundary. In this paper, the sign-definiteness condition guarantees the existence of the limit describing the asymptotic growth rate. Our preliminary analysis suggests that this fixed-sign assumption is a technical artifact of the current proof rather than an intrinsic limitation of the method. For inverse source problems, it appears possible to remove the sign-definite condition. Another direction is to use the Paley-Wiener-type theorem to treat the more challenging case of phaseless data. Such an extension would enhance the practical applicability of our theoretical framework to experiments in which only intensity measurements are available.

Building on these uniqueness results, we are developing numerical algorithms for reconstructing the support of the potential. The Paley-Wiener-type theorem can also be used to reconstruct the potential under additional smoothness assumptions on \(q\). It is also possible to extend the uniqueness results to the Helmholtz equation $\Delta u + \kappa^2 q u = 0$; such an extension requires a careful analysis of the high-frequency asymptotics of the wave equation. 

\clearpage

\setcounter{section}{0}
\setcounter{equation}{0}
\renewcommand{\thesection}{\Alph{section}}
\renewcommand{\theequation}{\thesection.\arabic{equation}}

\section*{Appendix}
\addcontentsline{toc}{section}{Appendix}

\section{Proof of \eqref{growth-pw}}
\label{sec:proof-of-growth-pw}

We first recall the following classical one-dimensional Paley-Wiener theorem; see, e.g., \cite[Chapter III D, pp.~30--37]{PK}.

\begin{lemma}
\label{lem:1d}
Let $v\in L^2(\mathbb R)$ have compact support. Then
$$
\sup\operatorname{supp}v
=
\limsup_{\eta\to+\infty}\frac1\eta\log|\hat v(i\eta)| ,
$$
where $\hat v$ is the Fourier transform of $v$ in dimension $1$.
\end{lemma}

Using this lemma, we prove the following proposition, which gives \eqref{growth-pw}.

\begin{proposition}
\label{ndpw}
Let $u\in L^2(\mathbb R^n)$ have compact support, and let \(e\in\mathbb S^{n-1}\). Set
$$
h:=\sup_{x\in\operatorname{supp}u}x\cdot e .
$$
Then
$$
h=
\sup_{\xi\in\mathbb R^n}
\limsup_{\eta\to+\infty}
\frac1\eta\log|\hat u(\xi+i\eta e)|,
$$
where $\hat u$ is the Fourier transform of $u$ in dimension $n$.
\end{proposition}

\begin{proof}
A compactly supported function \(u\in L^2(\mathbb R^n)\) is also a compactly supported distribution. Hence, by Proposition \ref{PW}, there exist $C,N>0$ such that
$$
|\hat u(\zeta)|\le C(1+|\zeta|)^N e^{H(\operatorname{Im}\zeta)},
\qquad \zeta\in\mathbb C^n,
$$
where
$$
H(\xi):=\sup_{x\in\operatorname{supp}u}x\cdot \xi.
$$
For $\zeta=\xi+i\eta e$, we have $\operatorname{Im}\zeta=\eta e$, hence
$$
H(\operatorname{Im}\zeta)=H(\eta e)=\eta H(e).
$$
Therefore
$$
|\hat u(\xi+i\eta e)|
\le C(1+|\xi|+\eta)^N e^{\eta H(e)},
$$
which gives
\begin{equation}\label{pwleq-L2}
\sup_{\xi\in\mathbb R^n}
\limsup_{\eta\to+\infty}
\frac1\eta\log|\hat u(\xi+i\eta e)|
\le H(e).
\end{equation}

We now prove that equality holds. Suppose, for a contradiction, that there exists $a<h$ such that
$$
\limsup_{\eta\to+\infty}\frac1\eta\log|\hat u(\xi+i\eta e)|\le a
\qquad\text{for all }\xi\in\mathbb R^n.
$$
Fix $\xi\in\mathbb R^n$ and define a one-dimensional function $v_\xi\in L^2_{\mathrm{comp}}(\mathbb R)$ by
$$
\langle v_\xi,\varphi\rangle
:=
\int_{\mathbb R^n} u(x)\varphi(x\cdot e)e^{-ix\cdot\xi}\,dx,
\qquad \varphi\in C^\infty(\mathbb R).
$$
Equivalently, $v_\xi$ is the compactly supported $L^2$ function whose Fourier transform satisfies
$$
\hat v_\xi(\tau)
=
\int_{\mathbb R^n} u(x)e^{-i\tau x\cdot e}e^{-ix\cdot\xi}\,dx
=
\hat u(\xi+\tau e).
$$
In particular,
$$
\hat v_\xi(i\eta)=\hat u(\xi+i\eta e).
$$
Since $\operatorname{supp}u\subset\{x\cdot e\le h\}$, we have
$$
\operatorname{supp}v_\xi\subset(-\infty,h].
$$
By Lemma~\ref{lem:1d},
$$
\sup\operatorname{supp}v_\xi
=
\limsup_{\eta\to+\infty}\frac1\eta\log|\hat v_\xi(i\eta)|
=
\limsup_{\eta\to+\infty}\frac1\eta\log|\hat u(\xi+i\eta e)|
\le a.
$$
Hence $v_\xi=0$ on $(a,\infty)$, i.e.,
\begin{equation}\label{pweq-L2}
\int_{\mathbb R^n} u(x)\varphi(x\cdot e)e^{-ix\cdot\xi}\,dx=0
\end{equation}
for every $\varphi\in C_{\mathrm{c}}^\infty((a,\infty))$ and every $\xi\in\mathbb R^n$.

Now fix $\varphi\in C_{\mathrm{c}}^\infty((a,\infty))$ and set
$$
\psi(x):=\varphi(x\cdot e)u(x).
$$
Then $\psi\in L^2_{\mathrm{comp}}(\mathbb R^n)$ and by \eqref{pweq-L2} we have
$$
\hat\psi(\xi)
=
\int_{\mathbb R^n} u(x)\varphi(x\cdot e)e^{-ix\cdot\xi}\,dx
=0
\qquad\text{for all }\xi\in\mathbb R^n.
$$
By Fourier inversion, $\psi=0$, so
$$
\varphi(x\cdot e)u(x)=0
$$
as an $L^2$ function on $\mathbb R^n$. Since $\varphi\in C_{\mathrm{c}}^\infty((a,\infty))$ is arbitrary, it follows that
$$
u=0
\qquad\text{on }\{x:x\cdot e>a\}.
$$
Thus
$$
\operatorname{supp}u\subset\{x:x\cdot e\le a\},
$$
which implies $h\le a$, contradicting $a<h$. Therefore, equality holds in \eqref{pwleq-L2}, that is,
$$
\sup_{\xi\in\mathbb R^n}
\limsup_{\eta\to+\infty}
\frac1\eta\log|\hat u(\xi+i\eta e)|
= h.
$$
\end{proof}

\section{Proof of Lemma \ref{thm:LS-equivalence}}
\label{sec:proof-of-ls-equivalence}

We prove (i)--(ii) in turn.
    
\vskip 5pt
(i)  Assume that $u\in H^2_{\mathrm{loc}}(\mathbb{R}^3) \cap C^1\big( \mathbb R^3 \backslash B_R\big)$ satisfies \eqref{eqs}-\eqref{eq:radiation}. Write \(\varPhi_\kappa(\cdot) := \varPhi(\cdot; \kappa)\). 
Denote by $\nu$ the outward unit normal on $\partial B_R$. 

Note that \(\Delta u + \kappa^2 u = - q u\) and
$\Delta\varPhi_\kappa + \kappa^2\varPhi_\kappa = -\delta_0$. Applying Green's second identity on \(B_R\) to \(u\) and a smooth approximation of
$\varPhi_\kappa(\cdot - x)$, we obtain 
\begin{equation}\label{eq:green-rep}
u(x)=\int_{B_R}\varPhi_\kappa(\cdot-x) q u dy
-\int_{\partial B_R}\left(
u\frac{\partial\varPhi_\kappa(\cdot-x)}{\partial\nu}
-\varPhi_\kappa(\cdot-x) \frac{\partial u}{\partial\nu}
\right) ds
~\text{ for a.e. }  x \in B_R.
\end{equation}
Write $u=u^{\mathrm{in}}+u^{\mathrm{sc}}$. For the incident field
$u^{\mathrm{in}}$, which satisfies
$(\Delta+\kappa^2)u^{\mathrm{in}}=0$,  Green's formula gives
\begin{equation}\label{eq:green-inc}
u^{\mathrm{in}}(x)
=\int_{\partial B_R}\left(
\frac{\partial u^{\mathrm{in}}}{\partial\nu}\varPhi_\kappa(\cdot-x)
-u^{\mathrm{in}}\frac{\partial\varPhi_\kappa(\cdot-x)}{\partial\nu}
\right) ds,
~ x \in B_R.
\end{equation}
For the scattered field $u^{\mathrm{sc}}$, applying Green's formula in the
exterior domain and using the radiation condition \eqref{eq:radiation} yields
\begin{equation}\label{eq:green-sc}
\int_{\partial B_R}\left(
\frac{\partial u^{\mathrm{sc}}}{\partial\nu}\varPhi_\kappa(\cdot-x)
-u^{\mathrm{sc}}\frac{\partial\varPhi_\kappa(\cdot-x)}{\partial\nu}
\right) ds=0,
~ x \in B_R.
\end{equation}
Combining \eqref{eq:green-inc} and \eqref{eq:green-sc}, we find
$$
\int_{\partial B_R}\left(
u\frac{\partial\varPhi_\kappa(\cdot-x)}{\partial\nu}
-\varPhi_\kappa(\cdot-x)\frac{\partial u}{\partial\nu}
\right) ds
=-u^{\mathrm{in}}(x), ~ x \in B_R.
$$
Substituting this into \eqref{eq:green-rep} gives
$$
u(x)=e^{i\kappa x\cdot d}+\int_{B_R}\varPhi_\kappa(\cdot-x) q u dy,~ x \in B_R.
$$
Thus, \eqref{LS} holds for \(u\).

\vskip 5pt
(ii) Assume that $u\in L_{\mathrm{loc}}^2(\mathbb{R}^3)$ satisfies \eqref{LS}. Define
\begin{align}\label{eq-lem:ScatteringField-20261004}
    u^{\mathrm{sc}}(x):=\int_{B_R}\varPhi(x-y; \kappa)q(y)u(y) dy, 
    ~ x\in \mathbb R^3. 
\end{align}
Since $\varPhi(\cdot ;\kappa)$ is the fundamental solution for the operator $-\Delta - \kappa^2$ and supp$(q) \subset B_R $, the above implies 
\begin{align*}
    (-\Delta -\kappa^2)u^{\mathrm{sc}}=q u  ~\text{ in }  L^2( \mathbb R^3). 
\end{align*}
Therefore, the solution $u(x) = e^{i\kappa x\cdot d} + u^{\mathrm{sc}}(x)$ (see \eqref{LS}) satisfies 
$$
(\Delta+\kappa^2)u
=(\Delta+\kappa^2)e^{i\kappa x\cdot d}
+(\Delta+\kappa^2)u^{\mathrm{sc}}
=0-q u=-q u
~\text{ in } L^2(\mathbb R^3),
$$
i.e., $u$ satisfies \eqref{eqs}. By the ellipticity of the operator $\Delta$, one can check that $u \in H^2_{\mathrm{loc}}(\mathbb R^3)$ and it is smooth outside supp$(q)$. 

Next, \eqref{fm} gives, for \(y \in B_R\) and $\hat x := x/|x|$,
\begin{align}\label{AsymptoticsOfFundamentalSolution-20261004}
    \varPhi(x-y; k)
    =\frac{ e^{i k |x-y|} }{4\pi|x-y|}
    =\frac{ e^{i k |x|} }{4\pi|x|}
    \left( e^{-i k \hat x\cdot y}+\mathcal{O}(|x|^{-1})\right) 
    \quad\text{as }|x|\to+\infty,
\end{align}
where \(\mathcal{O}(|x|^{-1})\) denotes a remainder whose magnitude and first-order derivatives are bounded by a constant multiple of \(|x|^{-1}\) as \(|x| \rightarrow +\infty\). 
Then, we see from \eqref{eq-lem:ScatteringField-20261004} and \eqref{AsymptoticsOfFundamentalSolution-20261004} (with $k=\kappa$) that 
$$
u^{\mathrm{sc}}(x)
=\frac{e^{i\kappa|x|}}{|x|}
\left(
\frac{1}{4\pi}\int_{\mathbb{R}^3}
e^{-i\kappa\hat x\cdot y}q(y)u(y)\,\mathrm{d}y
+\mathcal{O}(|x|^{-1})
\right)
\quad\text{as }|x|\to+\infty.
$$
This shows that \(u^{\mathrm{sc}}\) satisfies the Sommerfeld radiation condition \eqref{eq:radiation} and that its far-field pattern is given by \eqref{FP}. This completes the proof.

\vspace{1em}

\textbf{Acknowledgments}
\vspace{0.5em}

The work of Guanghui Hu was supported by the National Natural Science Foundation of China (No. 12425112), the Fundamental Research Funds for Central Universities in China (No. 63263073 and 63261173) and a Key Program (No. 25JCZDJC00970) of the Natural Science Foundation of Tianjin, China. The work of Yubiao Zhang was partially supported by the National Natural Science Foundation of China under Grant No. 12671535. The work of Jun Zou was substantially supported by Hong Kong RGC General Research Fund (projects 14311125 and 14306623) and NSFC/Hong Kong RGC Joint Research Scheme (project N\_CUHK465/22).

\end{document}